\documentclass[a4paper,11pt]{article}

\usepackage[english]{babel} 
\usepackage[latin1]{inputenc}
\usepackage[T1]{fontenc}
\usepackage{dsfont}
\usepackage{graphicx}
\usepackage{amsmath, amssymb, amsfonts,amsthm}
\usepackage{stmaryrd}
\usepackage{lmodern} 
\usepackage{cite}
\usepackage{pstricks-add}
\usepackage{a4wide}
\usepackage{url}
\usepackage[ruled,vlined]{algorithm2e}
\usepackage{color}
\usepackage{authblk}
\usepackage{multirow}

\providecommand{\SetAlgoLined}{\SetLine} 
\providecommand{\DontPrintSemicolon}{\dontprintsemicolon}
\providecommand{\LinesNumbered}{\linesnumbered}

\theoremstyle{plain}
\newtheorem{theorem}{Theorem}
\newtheorem{assumption}{Assumption}
\newtheorem{lemma}{Lemma}

\newtheorem{proposition}[lemma]{Proposition}
\theoremstyle{definition}
\newtheorem{definition}{Definition}

\theoremstyle{remark}

\newcommand{\argmin}{\operatorname{argmin}}
\newcommand{\argmax}{\operatorname{argmax}}
\newcommand{\paren}[1]{\left( \left. #1 \right. \right)} 
\newcommand{\croch}[1]{\left[ \left. #1 \right. \right]} 
\newcommand{\set}[1]{\left\{ \left. #1 \right. \right\}}
\newcommand{\1}{1\hskip-2.6pt{\rm l}}

\newcommand{\sachant}{\;\right|\left.} 

\newcommand{\R}{\mathbb{R}}

\newcommand{\Lbb}{\mathbb{L}}
\newcommand{\TVF}{\ensuremath{\mathrm{TVF}}}

\newcommand{\Acal}{\mathcal{A}}
\newcommand{\Bcal}{\mathcal{B}}

\newcommand{\Dcal}{\mathcal{D}}
\newcommand{\Ecal}{\mathcal{E}}
\newcommand{\Fcal}{\mathcal{F}}
\newcommand{\Gcal}{\mathcal{G}}
\newcommand{\Hcal}{\mathcal{H}}
\newcommand{\Ical}{\mathcal{I}}

\newcommand{\Kcal}{\mathcal{K}}
\newcommand{\Lcal}{\mathcal{L}}
\newcommand{\Mcal}{\mathcal{M}}

\newcommand{\Pcal}{\mathcal{P}}

\newcommand{\Scal}{\mathcal{S}}
\newcommand{\Tcal}{\mathcal{T}}

\newcommand{\Xcal}{\mathcal{X}}

\newcommand{\egaldef}{:=} 
\newcommand{\critVFCV}{\crit_{\mathrm{VFCV}}} 
\newcommand{\critTVF}{\crit_{\mathrm{TVF}}} 
\newcommand{\mM}{m\in\Mcal}

\newcommand{\propref}[1]{Proposition~\ref{prop:#1}} 
\newcommand{\secref}[1]{Section~\ref{sec:#1}} 
\newcommand{\subref}[1]{Section~\ref{sub:#1}} 
\newcommand{\tabref}[1]{Table~\ref{tab:#1}} 
\newcommand{\figref}[1]{Figure~\ref{fig:#1}} 

\RequirePackage[colorlinks,citecolor=blue,urlcolor=blue]{hyperref}

\newcommand{\crit}{\mathop{\rm crit}\nolimits}

\newcommand{\Xbf}{\mathbf{X}}
\newcommand{\scal}[2]{\langle #1,#2\rangle}

\begin{document}
\title{A new V-fold type procedure based on robust tests}%
\author[1]{Lucien Birg\'e\thanks{lucien.birge@upmc.fr}}
\author[2,3]{Nelo Magalh\~aes\thanks{nelo.moltermagalhaes@gmail.com}}
\author[2,3]{Pascal Massart\thanks{pascal.massart@math.u-psud.fr}}
\affil[1]{LPMA,  UPMC Universit\'{e} Paris 06}
\affil[2]{\'{E}quipe Probabilit\'{e}s et Statistiques, Universit\'{e}  Paris-Sud 11}
\affil[3]{INRIA team Select}
\date{June 2015}

\maketitle

\begin{abstract}
We define a general V-fold cross-validation type method based on robust tests, which is an extension of the hold-out defined by Birg\'{e} \cite[Section 9]{Birge2006b}. We give some theoretical results showing that, under some weak assumptions on the considered statistical procedures, our selected estimator satisfies an oracle type inequality. We also introduce a fast algorithm that implements our method. Moreover we show in our simulations that this V-fold performs generally well for estimating a density for different sample sizes, and can handle well-known problems, such as binwidth selection for histograms or bandwidth selection for kernels. We finally provide a comparison with other classical V-fold methods and study empirically the influence of the value of $V$ on the risk.
\end{abstract}

\footnotetext{
{\it \hspace{2.3mm}Key words and phrases.} T-estimation, density estimation, V-fold cross-validation, Hellinger loss. 
}

\section{Introduction}\label{sec:I}
The purpose of this paper is to offer a new method to solve the following problem. Suppose we are given i.i.d.\ observations from an unknown distribution $P_s$ to be estimated. This distribution is often assumed to have a density $s$ with respect to some given measure $\mu$, hence our notation, but we shall also consider the case when $P_s$ is not absolutely continuous with respect to $\mu$, keeping the same notation $P_s$ for the true distribution, in which case the subscript $s$ just indicates that $P_s$ is the distribution of the observations. 

We also have at hand  a family of statistical procedures or algorithms $\{{\cal A}_m,\,m\in{\cal M}\}$ that can be applied to the observations in order to derive estimators of $P_s$. How can we use our data in order to choose one potentially optimal algorithm in the family, provided that a criterion of quality for the estimators has been chosen? Let us now be somewhat more precise.

\subsection{The problem of procedure choice}\label{sub:I1}
We observe an $n$-sample $\Xbf=\{X_{1},\ldots,X_{n}\}$ of random variables $X_i$ with values 
in the measured space $(\Xcal,\Ecal)$ and we assume (temporarily) that the distribution  $P_s=
s\cdot\mu$ of the $X_i$ admits a density $s$ with respect to some given positive measure $\mu$ 
on $\Xcal$ and that $s$ belongs to some given subset $\Scal$ of $\Lbb_1(\mu)$. The purpose here 
is to use the observations in order to design an estimator $\widehat{s}=\widehat{s}(\Xbf)$ of $s$.

There is a huge amount of strategies for solving this estimation problem, depending on the additional assumptions one makes about $s$. We shall use the notion of {\em statistical procedure} (procedure for short), also denoted {\em statistical algorithm} in what follows, in order to properly formalize these strategies. Following \cite{Arlot-Celisse2010}, we define a \textit{procedure} or an \textit{algorithm} as any measurable mapping $\Acal$ from $\bigcup_{k\geq1} \Xcal^k$ to $\Scal$. Such a procedure associates to any random sample $\mathbf{Y}_k\in\Xcal^k$ an estimator $\widehat{s}_k=\Acal(\mathbf{Y}_k)\in\Lbb_1(\mu)$ of $s$. A classical criterion from decision theory used to measure the quality of a procedure $\Acal$ based on an i.i.d.\ sample of size $k$ when $s$ obtains is its \textit{risk}: $\mathbb{E}_s\croch{\ell(s,\Acal(\mathbf{Y}_k))}$, where $\ell$ is some given \textit{loss function} and $\mathbb{E}_s$ denotes the expectation when $s$ obtains, i.e.\ when the distribution of $\mathbf{Y}_k$ is $P_s^{\otimes k}$. The smaller the risk, the better the procedure $\Acal$.

To define the risk of a procedure one can consider various loss functions. Some popular ones are derived from a \textit{contrast function} $\gamma$ (see \cite[Definition 1]{Birge-Massart1993}) which is a mapping from $\Scal \times \Xcal$ to $\R$ such that $s$ minimizes over $\Scal$ the function $t\mapsto\mathbb{E}_s\croch{\gamma(t,X)}$. The loss $\ell$ at $t$ is then defined as
\begin{equation}\label{eq:contrast-loss}
\ell(s,t)=\mathbb{E}_s\croch{\gamma(t,X)-\gamma(s,X)}\geq 0 \quad \text{for all }t\in \Scal,
\end{equation} 
hence $\ell(s,s)=0$. The $\Lbb_2$-loss derives from the choice $\Scal=\Lbb_2(\mu)\cap\Lbb_1(\mu)$ and 
$\gamma(t,x)=\|t\|^2-2 t(x)$, where $\|t\|=\left[\int_{\Xcal}t^2 d\mu\right]^{1/2}$ denotes the $\Lbb_2$-norm. 
The Kullback-Leibler loss corresponds to the contrast function $\gamma(t,x)=-\log(t(x))$ with $\Scal$ being the set of all probability densities with respect to $\mu$. 

In this paper, we consider the problem of \textit{procedure selection}. Let $(\Acal_{m})_{\mM}$ denote a collection of candidate statistical procedures. Our goal is to choose from the observations $\Xbf$ one of these procedures, that is some $\widehat{m}(\Xbf)\in \Mcal$, in order to have the most accurate estimation of $s$. If we apply all these procedures to the sample $\Xbf$ we get the corresponding collection of estimators $\{\widehat{s}_m=\Acal_m(\Xbf),\,\mM\}$. Given a loss $\ell$, the best possible choice for $m$ would be to select $m^*\in \Mcal$ such that 
\[
\mathbb{E}_s\croch{\ell\left(s,\widehat{s}_{m^*}(\Xbf)\right)}=
\inf_{\mM} \mathbb{E}_s\croch{\ell\left(s,\widehat{s}_{m}(\Xbf)\right)}.
\]
Unfortunately, since $s$ is unknown, all the  risks $\mathbb{E}_s\croch{\ell(s,\widehat{s}_m)}$ are unknown as well and we cannot select the so-called \textit{oracle algorithm} $\Acal_{m^*}$. One can only hope to choose $\widehat{m}=\widehat{m}(\Xbf)$ in such a way that $\mathbb{E}_s\croch{\ell(s,\widehat{s}_{\widehat{m}})}$ is close to $\mathbb{E}_s\croch{\ell(s,\widehat{s}_{m^*})}$.\\

To make this presentation more explicit, let us mention some classical estimation problems that naturally fit into it:
\begin{itemize}
\item \textit{Bandwidth selection} (see\cite[Chapter 11]{Devroye-Lugosi2001}). Let $\Xcal=\R$, $\mu$ 
be the Lebesgue measure, $K:\R\rightarrow \R$ a given nonnegative function satisfying 
$\int_{\Xcal}K(x)\,dx=1$ and $\Hcal=\{h_m, \mM\}$ be a finite or countable set of positive 
bandwidths. We define the \textit{kernel algorithm} $\Acal_m$ as the procedure that produces from any sample $\mathbf{Y}_k$ of size $k$ a kernel density estimator with bandwidth $h_m$, which means that
\[
\Acal_m(\mathbf{Y}_k)(x)=\frac{1}{kh_m}\sum_{Y_i \in\mathbf{Y}_k} K\left(\frac{x-Y_i}{h_m}\right)\quad\mbox{for all }x\in\R.
\]
The problem of choosing a best estimator among the family $\{\widehat{s}_m,\,\mM\}$ amounts to select a ``best'' bandwidth in $\Hcal$, that is one that minimizes the risk $\mathbb{E}_s\croch{\ell(s,\widehat{s}_m)}$ with respect to $m$.
\item \textit{Model selection} (see \cite{Massart2007}). We recall that a \textit{model} $S$ for $s$ 
is any subset of $\Scal$. It follows from \eqref{eq:contrast-loss} that minimizing, for $t$ in $S$, 
the loss $\ell(s,t)$ derived from the contrast function $\gamma$ amounts to minimizing 
$t\mapsto\mathbb{E}_s\croch{\gamma(t,X)}$ over $S$. Since $s$ is unknown, this is impossible but if we replace $\mathbb{E}_s\croch{\gamma(t,X)}$ by its unbiased empirical version: $\gamma_n(t)=n^{-1}\sum_{i=1}^n\gamma(t,X_i)$ we can derive an
estimator with values in $S$ by minimizing $\gamma_n(t)$ over $S$ instead. This procedure $\Acal_S$ is a \textit{minimum contrast algorithm} that provides a \textit{minimum contrast estimator} $\widehat{s}_S(\Xbf)\in\argmin_{t\in S}\gamma_n(t)$ on $S$. Using for instance, the Kullback-Leibler contrast on  a set $S$ of densities leads to the so-called ``maximum likelihood estimator'' on $S$.

If we have at hand some finite or countable collection of models $\{S_m\}_{\mM}$ and a suitable contrast function $\gamma$ we may associate in this way to each model $S_m$ a minimum contrast algorithm $\Acal_m$ and the corresponding minimum contrast estimator $\widehat{s}_m(\Xbf)$. The problem of ``model selection'' is to select from the data a ``best model'' (one with the minimal risk) in the family, leading to a ``best'' possible minimum contrast estimator. 
\end{itemize}

Instead of deriving the loss function $\ell$ from a contrast function we may use for $\ell$ the squared Hellinger distance provided that our estimators $\widehat{s}_m$ are genuine probability densities. We recall that the \textit{Hellinger distance} $h$ and the {\em Hellinger affinity} $\rho$ between two probabilities $P$ and $Q$ defined on $\Xcal$ are given respectively by
\begin{equation}\label{eq:hell-rho}
h(P,Q)=\left[\frac{1}{2}\int{\paren{\sqrt{dP} - \sqrt{dQ}}^{2}}\right]^{1/2}\quad\mbox{ and }\quad
\rho(P,Q)=\int\sqrt{dPdQ}=1-h^2(P,Q),
\end{equation}
where $dP$ and $dQ$ denote the densities of $P$ and $Q$ with respect to any dominating measure (the result being independent of this choice). One advantage of this loss function lies in the fact that 
$h$ is a distance on the set $\Pcal $ of {\em all} probabilities on $\Xcal$ and therefore does not require that $P_s$ be absolutely continuous with respect to $\mu$, which is one of the reasons why we shall use it in the sequel. In this case we take for $\Scal$ a set of probability densities with respect to $\mu$ and we set, for all $t$ in $\Scal$ and $P_t=t\cdot\mu$, $\ell(s,t)=h^2(P_s,P_t)$ which we shall write $h^2(s,t)$ for simplicity. We shall also write $\rho(t,u)$ for $\rho(P_t,P_u)$. This loss then leads to the \textit{quadratic Hellinger risk}. 

\subsection{Cross-validation}\label{sub:I2}
The biggest difficulty for selecting a procedure in a given family $\{\Acal_m,\,\mM\}$ comes from the fact that we use the same data $\Xbf$ to build the estimators $\widehat{s}_m(\Xbf)$ and to evaluate their quality. It is indeed well-known that evaluating the statistical performance of a procedure with the same data that have been used for the construction of the corresponding estimator leads to an overoptimistic result. One solution to avoid this drawback is to save a fraction of the initial sample to test the output of the procedures $\Acal_m$ on it. This is the basic idea behind {\em cross-validation} (CV) which relies on data splitting. 

The simplest CV method is the \textit{hold-out} (HO) which corresponds to a single split of the data. The set $\Xbf$ is divided once and for all into two non-empty proper subsets $\Xbf^{t}$ and $\Xbf^{v}=\Xbf\setminus \Xbf^{t}$ to be called respectively the \textit{training} and the \textit{validation} sample. First, with the training sample $\Xbf^{t}$, we construct a set $\{\Acal_m(\Xbf^{t}),\,\mM\}$ of preliminary estimators. Then, using the validation sample $\Xbf^{v}$, we choose a criterion in order to evaluate the quality of each procedure $\Acal_m$ from the observation of $\Acal_m(\Xbf^{t})$. Finally, we select $\widehat{m}(\Xbf^{v})$ minimizing this criterion over $\Mcal$. Depending on the author, the final estimator might be either $\Acal_{\widehat{m}}(\Xbf^{t})$ (as in \cite{Devroye-Lugosi2001}) or $\Acal_{\widehat{m}}(\Xbf)$ (as in \cite{Arlot-Lerasle2014}). All CV methods are deduced from the HO: instead of using one single partition of our sample, we use different partitions, compute the HO criterion for each one and finally define the CV criterion by averaging all the HO criteria. The goal, by considering several partitions instead of one, is to reduce the variability with the hope that the CV criterion  will lead to a more accurate evaluation of the quality of each procedure.

We shall focus here on V-fold cross-validation (VFCV) which corresponds to a particular set of data splits\footnote{The concerned reader should have a look at the survey of Arlot and Celisse \cite{Arlot-Celisse2010} to get a complete overview of other CV methods.}. One divides the sample $\Xbf$ into $V\ge2$ disjointed and therefore independent subsamples $\Xbf_{j}$, $j=1,\ldots,V$, of the same size $p=n/V$ (assuming, for simplicity, that $p$ is an integer) so that $\Xbf=\bigcup_{j=1}^{V} \Xbf_{j}$.  For each split $j \in\{1,\ldots,V\}$, one uses $\Xbf_{j}^{c}$ to build the family of ``partial estimators'' $\{\widehat{s}_{m,j}=\Acal_{m}(\Xbf_{j}^{c}),\,\mM\}$ and the corresponding validation sample $\Xbf_{j}$ to define an evaluation criterion $\crit_j(m)=\crit_j(m)(\Xbf_{j})$ of the procedure $\Acal_m(\Xbf_{j}^{c})$ corresponding to the partition $(\Xbf_{j},\Xbf_{j}^c)$ of the data. One finally selects a strategy $\widehat{m}_{{\rm VF}}$ minimizing the averaged criterion:
\begin{equation*}
\widehat{m}_{{\rm VF}} \in \underset{\mM}{\operatorname{argmin}} \crit(m)
\qquad\mbox{with}\qquad \crit(m)=\frac{1}{V}\sum_{j=1}^V \crit_j(m).
\end{equation*}

There are as many V-fold procedures as there are different ways to define $\crit_j(m)$. If we work 
with a loss of the type \eqref{eq:contrast-loss}, the best estimator in the family $\{\widehat{s}_{m,j},\,\mM\}$ is the one minimizing the loss, i.e.\ the one minimizing $\mathbb{E}_s\croch{\gamma(\widehat{s}_{m,j},X)}$ (with $X$ being independent of $\Xbf_j^c$). A natural idea for evaluating this quantity that we cannot compute since we do not know $s$ is to estimate it by its empirical version based on the independent sample $\Xbf_j$ of size $p$, which leads to the criterion
\begin{equation*}
\crit_j(m)= \frac{1}{p} \sum_{X_i\in \Xbf_j} \gamma\paren{\widehat{s}_{m,j},X_i}.
\end{equation*}
In this classical context, we naturally select the statistical procedure with the lowest estimated average loss $\crit(m)$.
The choice $\gamma(t,x)=-\log(t(x))$ leads to the Kullback-Leibler V-fold (KLVF) whereas $\gamma(t,x)=\|t\|^2-2 t(x)$ provides the Least-Squares V-fold (LSVF). The chosen estimators will be respectively denoted $\widehat{m}_{{\rm KLVF}}$ and $\widehat{m}_{{\rm LSVF}}$ and the relevant classical criterion will be denoted $\critVFCV$ in  what follows. 

\subsection{An alternative criterion}\label{sub:I3}
When the chosen loss function that we use is the squared Hellinger distance, an alternative empirical criterion to evaluate the quality of an estimator has been proposed by Birg\'e \cite{Birge1983} following ideas of Le Cam \cite{LeCam1973,LeCam1975} to process estimator selection. An alternative method was later introduced by Baraud \cite{Baraud2011}. An HO strategy based on this criterion was first proposed by Birg\'e in \cite{Birge2006b}, this latter procedure being recently implemented in \cite{Magalhaes-Rozenholc2014}. The idea behind the construction is as follows. Suppose we have at hand a set $\Tcal$ of densities with respect to $\mu$ and, for each pair $(t,u)$, $t\ne u$, of points of $\Tcal$, a test $\psi_{t,u}$ between $t$ and $u$ ($\psi_{t,u}=t$ meaning accepting $t$). Given a sample $\Xbf$ we may perform all the tests $\psi_{t,u}(\Xbf)$ and consider the criterion $\Dcal(t)$ defined on $\Tcal$ by
\begin{equation}\label{eq:crit robust}
\Dcal(t)= \sup_{u\in\Tcal,\:u\ne t} h(t,u)\1_{\set{\psi_{t,u}(\Xbf)=u}}.
\end{equation}
It immediately follows from this definition that
\begin{equation}
h(t,u)\le\max\{\Dcal(t),\Dcal(u)\}\quad\mbox{for all }t,u\in\Tcal .
\label{eq:control1}
\end{equation}

This definition means that $\Dcal(t)$ is large when there exists some $u$ which is far from $t$ and which is preferred to $t$ by the test $\psi_{t,u}(\Xbf)$, suggesting that $t$ is likely to be far from $s$, at least if $s$ does belong to $\Tcal$. In order that this be actually true even if $P_s$ does not belong to $\{P_t, t\in\Tcal\}$, it is necessary to design suitable tests. It has been shown in \cite{Birge1983} that one can build a special test $\psi_{t,u}$ between the two Hellinger balls $\Bcal(t,r)$ and $\Bcal(u,r)$ with $r<h(t,u)/2$ (where $\Bcal(t,r)$ denotes the closed ball of center $t$ and radius $r$ in the metric space $(\Pcal,h)$) which posesses the required properties.  With this special choice of tests $\psi_{t,u}$ for all pairs $(t,u)$, $\Dcal(t)$ becomes indeed a good indicator of the quality of $t$ as an estimator of $s$ (the smaller $\Dcal(t)$, the better $t$) and, more generally, of $P_t$ as an estimator of $P_s$ even if $P_s$ is not absolutely continuous with respect to $\mu$. This property of $\Dcal$ suggests to define the following criterion on which to base a new VFCV procedure.
Starting from the family of preliminary density estimators 
\[
\set{\widehat{s}_{m,j}=\Acal_{m}(\Xbf_{j}^{c}),\,\mM,\;1\le j\le V},
\]
we build all the corresponding tests $\psi_{\widehat{s}_{l,j},\widehat{s}_{m,j}}(\Xbf_{j})$, hereafter denoted for simplicity by $\psi_{l,m}(\Xbf_{j})$, between the densities $\widehat{s}_{l,j}$ and $\widehat{s}_{m,j}$ for $l,m\in \Mcal$, $l\ne m$. Then, for each $j$ and $m$, we define the criterion $\crit_j(m)$ by
\begin{equation}\label{eq:crit THO}
\crit_j(m)=\Dcal_j^2(m)\quad\mbox{with}\quad\Dcal_j(m)=
\sup_{l\in \Mcal,\:l\ne m} h(\widehat{s}_{l,j},\widehat{s}_{m,j})\1_{\set{\psi_{l,m}(\Xbf_{j})=l}}.
\end{equation}
We then naturally define our test-based V-fold criterion as 
\begin{equation*}
\critTVF(m)\egaldef \overline{\Dcal}^2(m)= \frac{1}{V} \sum_{j=1}^V \Dcal^2_j(m)\quad\mbox{for all }\mM.
\end{equation*}
Up to our knowledge, this is the first V-fold type procedure based on the Hellinger distance. Note that this construction requires that the estimators $\widehat{s}_{m,j}$ be genuine probability densities with respect to $\mu$ which we shall assume from now on.

\subsection{Organization of the paper}\label{sub:I4}
Our goal is to study our new VFCV procedure from both a theoretical and a practical point of view. 
Section~2 is dedicated to its theoretical study. In Section~3 we discuss in details the implications of the resulting risk bounds to the case of histogram estimators, applications to kernel estimators and an extension to algorithms that do not lead to genuine probability density estimators.
Section~4 contains an empirical study of the influence of the value of $V$ on the performance of our procedure in terms of Hellinger risk and also comparisons with classical V-fold and some especially calibrated procedures. Section~5 describes the fast algorithm that we have designed and implemented in order to compute the selected estimator efficiently. Finally Section~6 contains a proof of the bounds for the Hellinger risk of kernel estimators. We provide some additional simulations in \secref{SuppMat}.
\section{T-V-fold}\label{sec:T}
As already mentioned, the method proposed in \cite{Birge2006b} is based on tests and it results in what Birg\'e called T-estimators (T for ``test''). We shall therefore call our cross-validation method based on the same tests T-V-fold cross-validation (TVF for short). 

\subsection{Tests between Hellinger balls}\label{sub:T1}
The tests that we use for our procedure satisfy the following assumption, which ensures their robustness. We recall that $\Scal$ is the set of all probability densities with respect to $\mu$.

\begin{assumption}[\bf TEST]\label{assum:A-Test}
Let $\theta\in (0,1/2)$ be given. For all $t$ and $u$ in $\Scal$, $z\in \R$ and $r=\theta h(t,u)$ there exists some test statistic $T_{t,u,\theta}(\Xbf)$ depending on $t,u,\theta$ and $\Xbf$ with the following properties. The test $\psi_{t,u}$ between $t$ and $u$ defined by
\begin{equation}\label{eq:generic test}
\psi_{t,u}(\Xbf)=\left\{
\begin{array}{lll}
        t \quad \text{if} \quad T_{t,u,\theta}(\Xbf)>z \\
        u\quad \text{if} \quad T_{t,u,\theta}(\Xbf)<z
        \end{array}
\right.,\quad z\in\R,
\end{equation}
with an arbitrary choice when $T_{t,u,\theta}(\Xbf)=z$, satisfies
\begin{equation}
\sup_{ \set{P_s \in \Pcal\,| \,h(s,t)\le r}} \mathbb{P}_s\croch{\psi_{t,u}(\Xbf)=u}\le\exp\croch{-n(1-2\theta)^2h^{2}(t,u)+z}
\label{eq:robustness1}
\end{equation}
and
\begin{equation}
\sup_{ \set{P_s \in \Pcal\,| \,h(s,u)\le r}} \mathbb{P}_s\croch{\psi_{t,u}(\Xbf)=t}\le\exp\croch{-n(1-2\theta)^2h^{2}(t,u)-z},
\label{eq:robustness2}
\end{equation}
where $\mathbb{P}_s$ denotes the probability that gives $\Xbf$ the distribution $P_s^{\otimes n}$.
\end{assumption}
Any test satisfying \eqref{eq:robustness1} and \eqref{eq:robustness2} will be suitable for our needs. 
\paragraph{Tests between balls}
In order to define tests between two Hellinger balls $\Bcal(t,r)$ and $\Bcal(u,r)$ with $r=\theta h(t,u)$, $0<\theta<1/2$, Birg\'e introduced the following test statistic
\begin{equation}\label{eq:test_birge}
T_{t,u,\theta}(\Xbf)=\sum_{i=1}^n\log\paren{\frac{\sin(\omega (1-\theta)) \sqrt{t}(X_i)+\sin(\omega \theta ) \sqrt{u}(X_i)}{\sin(\omega (1-\theta)) \sqrt{u}(X_i)+\sin(\omega \theta ) \sqrt{t}(X_i)}}\quad\mbox{with }
\omega=\arccos\rho(t,u).
\end{equation}
We should notice that for $\theta=0$, the test given by \eqref{eq:test_birge} is exactly the likelihood ratio test between $t$ and $u$. The fact that Assumption~(TEST) holds for this test whatever $\theta\in(0,1/2)$ has been proven in \cite{Birge1984b} and a more up-to-date version is to be found in \cite[Corollary 1]{Birge2013a}.
\subsection{TVF estimators}\label{sub:T2}
Let $(\Delta_m)_{\mM}$ denote some collection of positive numbers satisfying
\begin{equation} 
\Delta_m \geq 0\quad\mbox{for all }\mM, \qquad \text{and}\qquad\frac{1}{2}\le\Gamma=\sum_{\mM} \exp(-\Delta_m)< \infty. 
\label{eq:Gamma} 
\end{equation} 
Starting from the family of estimators $\widehat{s}_{m,j}$ defined in \subref{I3}, we consider the corresponding tests $\psi_{l,m}(\Xbf_j)=\psi_{\widehat{s}_{l,j},\widehat{s}_{m,j}}(\Xbf_{j})$ with $t= \widehat{s}_{l,j}$, $u= \widehat{s}_{m,j}$ and $z=\Delta_l-\Delta_m$ in \eqref{eq:generic test}.
This results in the estimator $\widehat{s}_{\widehat{m}_{\rm TVF}}$ derived from the procedure $\Acal_{\widehat{m}_{\rm TVF}}$ with
\begin{equation}\label{eq:critere TVF}
\widehat{m}_{\rm TVF} \in \underset{\mM}{\operatorname{argmin}} ~ \overline{\Dcal}^2(m)=
\underset{\mM}{\operatorname{argmin}} ~ \frac{1}{V} \sum_{j=1}^V \Dcal^2_j(m).
\end{equation}

\subsection{Assumption on the family of procedures}\label{sub:T3}
The idea of V-fold relies on the heuristic that, for each procedure $\Acal_m$, the observation of $V$ partial estimators $\widehat{s}_{m,j}$, $1\le j\le V$ based on samples of size $n-p$ with $p=n/V$ allows to predict the behavior of an estimator $\widehat{s}_m$ based on an $n$-sample. This requires that there exists a link between the loss of $\widehat{s}_m$ and the losses of the $\widehat{s}_{m,j}$. We shall need the following assumption on the collection of procedures we consider. 
%
\begin{assumption}[\bf LOSS]\label{assum:A-Loss}
For all procedures $\Acal_m$ with $\mM$, the loss at $s$ satisfies
\begin{equation*}
h^2\paren{s,\widehat{s}_m} \leq \frac{1}{V} \sum_{j=1}^V h^2\paren{s,\widehat{s}_{m,j}}.
\end{equation*}
\end{assumption}
This implies in particular that $R(\Acal_m,n,s)\le R(\Acal_m,n-p,s)$, where 
\[
R(\Acal,k,s)=\mathbb{E}_s\croch{h^2\paren{s,\Acal(\mathbf{Y}_k)\strut}}
\]
denotes the risk at $s$ of the procedure $\Acal$ based on a sample of size $k$. Assumption~(LOSS) is in particular satisfied by the ``additive estimators'' of \cite[Chapter 10]{Devroye-Lugosi2001}.
\begin{definition}
An \textit{additive estimator} $\widehat{s}=\widehat{s}(\Xbf)$ derived from a sample $\Xbf$ of size $n$ is an estimator that can be written in the form:
\begin{equation}\label{eq:linear estim}
\widehat{s}(x)=  \frac{1}{n} \sum_{i=1}^n~\Kcal(x,X_i) \quad\mbox{for all }x\in\Xcal,
\end{equation}
where $\Kcal$ is a real valued function from $\Xcal\times\Xcal$ to $\R$.
\end{definition}
There is a huge amount of literature about these estimators which already appeared in an early version in Whittle \cite{Whittle1958}. The first results about their asymptotic properties in general were made by Watson and Leadbetter \cite{Watson-Leadbetter1965}, followed by Winter \cite{Winter1975} and Walter and Blum \cite{Walter-Blum1979} who established rates (the latter authors called them \textit{delta sequence density estimators}). They were introduced in the context of CV by Rudemo \cite{Rudemo1982} and used by Marron for comparison of CV techniques \cite{Marron1987}. As shown in \cite{Walter-Blum1979} and \cite{Devroye-Lugosi2001}, additive estimators include in particular:
\begin{itemize}
\item \textit{Histogram estimators}. Given a partition $\{I_\lambda,\,\lambda \in \Lambda\}$ of $\Xcal$ with $0<\mu(I_\lambda)<+\infty$ for all $\lambda$ one defines the histogram estimator based on this partition as
\begin{equation}\label{eq:histo}
\widehat{s}(x)=\sum_{\lambda \in \Lambda}\left(\frac{1}{n}\sum_{i=1}^n\mathds{1}_{I_\lambda}(X_i)\right)\frac{\mathds{1}_{I_\lambda}(x)}{\mu(I_\lambda)}.
\end{equation}
It corresponds to the case of $\Kcal(x,X_i)=\sum_{\lambda \in \Lambda}[\mu(I_\lambda)]^{-1}
\mathds{1}_{I_\lambda}(X_i)\mathds{1}_{I_\lambda}(x)$.
\item \textit{Parzen kernel estimators on the line}. Set $\Kcal(x,X_i)=w^{-1}K\paren{w^{-1}(X_i-x)}$ for a given nonnegative kernel $K$ with $\int_\R K(x)\,dx=1$ and a positive bandwidth $w$. Then (\ref{eq:linear estim}) leads to a density estimator with respect to the Lebesgue measure on $\R$.
\end{itemize}
It is straightforward to check that if the procedure $\Acal_m$ results in additive estimators, the following relationship which says that the estimator built with the whole sample is exactly the convex combination of the $V$ partial estimators holds:
\begin{equation}
\widehat{s}_m = \frac{1}{V} \sum_{j=1}^V\widehat{s}_{m,j}.
\label{Eq-linear}
\end{equation}
As a consequence, we get the following elementary property:
\begin{proposition}\label{prop:MEAN}
Any procedure $\Acal_m$ which results in additive estimators does satisfy Assumption~(LOSS).
\end{proposition}
\begin{proof}
It follows from (\ref{Eq-linear}) and the concavity of the square root function that
\begin{equation*}
\rho(s,\widehat{s}_m)=\rho\paren{s,\frac{1}{V} \sum_{j=1}^V\widehat{s}_{m,j}}\ge\frac{1}{V} \sum_{j=1}^V \rho(s,\widehat{s}_{m,j}),
\end{equation*}
which is exactly Assumption~(LOSS) in view of (\ref{eq:hell-rho}).
\end{proof}
\subsection{The main result}
Assumption~(LOSS) ensures that, for the procedures we consider, the loss of some estimator is bounded by the mean of the losses of the partial estimators. This motivates us to work separately on each split $j\in\{1,\ldots,V\}$ and then to deduce a risk bound for the estimator built with the whole sample. It is therefore natural to study for each $j$ the deviations of the random variable $\Dcal_j(\cdot)$. A deviation inequality for $\Dcal$ has been proven in Theorem 9 of \cite{Birge2006b}. Let us now recall it and provide a short proof for the sake of completeness.
\begin{proposition}\label{prop:birge}
Let $(\Delta_m)_{\mM}$ be a collection of weights satisfying \eqref{eq:Gamma} and
\begin{equation*} 
A=\frac{n(1-2\theta)^2}{2V};\qquad y_{m,j}=\max\paren{\frac{h(s,\widehat{s}_{m,j})}{\theta},\sqrt{\frac{\Delta_m}{A}}}.
\end{equation*}
Then, for all $\mM$, and $j\in \{1,\ldots,V\}$,
\begin{equation*} 
\mathbb{P}_s\croch{\Dcal_j(m)\geq y \left|\,\Xbf_{j}^{c}\right.} \le\Gamma
\exp\croch{ -2Ay^2+\Delta_m} \qquad \text{for all } y\geq y_{m,j}.
\end{equation*}
\end{proposition}
\begin{proof}
Let us fix some $\mM$ and $j\in \{1,\ldots,V\}$ and work conditionally to the training sample $\Xbf_j^c$ 
so that the collection of estimators $(\widehat{s}_{l,j})_{l\in \Mcal}$ can be considered as fixed. We
perform the test $\psi_{l,m}(\Xbf_{j})$ that satisfy Assumption~(TEST) with $z=\Delta_l-\Delta_m$  in \eqref{eq:robustness1}. Then
\begin{eqnarray*}
\mathbb{P}_s\croch{\Dcal_j(m)\geq y \sachant \Xbf_{j}^{c}}
&=&\mathbb{P}_s\croch{\exists~ l\in \Mcal\mbox{ such that }h(\widehat{s}_{l,j},\widehat{s}_{m,j})\geq y
\mbox{ and }\psi_{l,m}(\Xbf_{j})=l \sachant \Xbf_{j}^{c} }\\
& \le& \sum_{l\in \Mcal:~ h(\widehat{s}_{l,j},\widehat{s}_{m,j})\geq y} \mathbb{P}_s\croch{\psi_{l,m}(\Xbf_{j})=l \sachant \Xbf_{j}^{c}} \\
& \le&\sum_{l\in \Mcal:~ h(\widehat{s}_{l,j},\widehat{s}_{m,j})\geq y} \exp\croch{ -2Ah^2\paren{\widehat{s}_{l,j},\widehat{s}_{m,j}}-\paren{\Delta_l-\Delta_m}}\\\
&\le&\exp\croch{ -2Ay^2+\Delta_m} \sum_{l\in \Mcal} \exp(-\Delta_l)\;\;\le\;\;
\Gamma\exp\croch{ -2Ay^2+\Delta_m},
\end{eqnarray*}
where we successively used the fact that $y\ge y_{m,j}\ge\theta^{-1}h(s,\widehat{s}_{m,j})$
and (\ref{eq:Gamma}).
\end{proof}

For each fixed $j$, that is conditionally to each $\Xbf_{j}^{c}$, we deal with some ``fixed geometrical configuration'' since the points $(\widehat{s}_{m,j})_{\mM}$ are given, conditionally to $\Xbf_{j}^{c}$. On this configuration, Proposition~\ref{prop:birge} controls the deviations of $\Dcal^2_j(m)$ which allows us to bound the expectation of $\overline{\Dcal}^2(m)$. This results in the following theorem. 
\begin{theorem}\label{thm:vfold.unrefined}
Under Assumption~(LOSS), the estimator $\widehat{s}_{\widehat{m}_{\TVF}}=\Acal_{\widehat{m}_{\TVF}}(\Xbf)$ with $\widehat{m}_{\TVF}$ minimizing the criterion $\overline{\Dcal}^2(m)$ satisfies the following inequality:
\begin{equation}
\mathbb{E}_s\croch{h^2\paren{s,\widehat{s}_{\widehat{m}_{\TVF}}} } \le \inf_{\mM} \set{2 \paren{ \frac{\theta^2+2}{\theta^2}}R\left(\Acal_m,\frac{V-1}{V}n,s\right)+
\frac{4V[\Delta_m+\log(2\Gamma)+1]}{n(1-2\theta)^2} }. 
\label{Eq-main}
\end{equation}
\end{theorem}
\begin{proof}
Let $m'$ be any point in $\Mcal$. It follows from (\ref{eq:control1}) that, for all $ \mM$ and $1\le j\le V$,
\begin{equation*}
h\paren{s,\widehat{s}_{m',j}} \leq h\paren{s,\widehat{s}_{m,j}} + h\paren{\widehat{s}_{m',j},\widehat{s}_{m,j}} \leq h\paren{s,\widehat{s}_{m,j}}  + \max\paren{\Dcal_j(m),\Dcal_j(m')}.
\end{equation*}
Setting $m'=\widehat{m}_{\TVF}=\widehat{m}$ for short, we derive that
\begin{align*}
\frac{1}{V} \sum_{j=1}^V  h^2\paren{s,\widehat{s}_{\widehat{m},j}}
& \leq 2    \set{\frac{1}{V} \sum_{j=1}^V h^2\paren{s,\widehat{s}_{m,j}} +\frac{1}{V} \sum_{j=1}^V \max\paren{\Dcal^2_j(m),\Dcal^2_j(\widehat{m})}}\\
& \leq 2  \set{\frac{1}{V} \sum_{j=1}^V h^2\paren{s,\widehat{s}_{m,j}} + \frac{1}{V} \sum_{j=1}^V (\Dcal^2_j(m)+\Dcal^2_j(\widehat{m})) }\\
& \leq \frac{2}{V} \sum_{j=1}^V h^2\paren{s,\widehat{s}_{m,j}} + 4 \overline{\Dcal}^2(m),
\end{align*}
for all $\mM$. Using Assumption~(LOSS) and taking expectations, we derive that
\begin{equation}
\mathbb{E}_s\croch{h^2\paren{s,\widehat{s}_{\widehat{m}}}}\le\frac{1}{V} \sum_{j=1}^V
\mathbb{E}_s\croch{h^2\paren{s,\widehat{s}_{\widehat{m},j}}}\le2R(\Acal_m,n-p,s)+
4\mathbb{E}_s\croch{\overline{\Dcal}^2(m)},
\label{Eq-aux1}
\end{equation}
since the risk of $\widehat{s}_{m,j}$ is the same for all $j$ and equal to $R(\Acal_m,n-p,s)$.

Let now $m$ and $j$ be fixed. Integrating the bound for $\mathbb{P}_s\croch{\Dcal_j^2(m)\geq y \sachant \Xbf_{j}^{c}}$ provided by \propref{birge} with respect to $y$ leads to
\[
\mathbb{E}_s\croch{\Dcal_j^2(m) \sachant \Xbf_{j}^{c}} \le y^2_{m,j}+\Gamma e^{\Delta_m}
\int_{y^2_{m,j}}^1e^{-2Az}\,dz\le y^2_{m,j}+\frac{\Gamma e^{\Delta_m}}{A}
\exp\left(-2Ay^2_{m,j}\right)
\]
and, since $Ay^2_{m,j}\ge\Delta_m$,
\[
\mathbb{E}_s\croch{\Dcal_j^2(m)}\le\mathbb{E}_s\croch{y^2_{m,j}}+\Gamma A^{-1}\exp(-\Delta_m)\le\frac{1}{\theta^2}\mathbb{E}_s\croch{h^2\paren{s,\widehat{s}_{m,j}}}+
\frac{\Delta_m+\Gamma e^{-\Delta_m}}{A}.
\]
Finally
\[
\mathbb{E}_s\croch{\overline{\Dcal}^2(m)}
\le\frac{1}{\theta^2}R(\Acal_m,n-p,s)+\frac{\Delta_m+\Gamma e^{-\Delta_m}}{A}.
\]
One should then observe that changing $\Delta_m$ into $\Delta_m+B$ with $B\ge0$
does not change the procedure since the tests only depend on differences 
$\Delta_m-\Delta_l$. Since the new weights $\Delta_m+B$ also satisfy (\ref{eq:Gamma})
with $\Gamma$ changed to $\Gamma e^{-B}$, the previous bound remains valid for the new weights leading to
\[
\mathbb{E}_s\croch{\overline{\Dcal}^2(m)}
\le\frac{1}{\theta^2}R(\Acal_m,n-p,s)+\frac{\Delta_m+B+\Gamma e^{-\Delta_m-2B}}{A}.
\]
An optimization with respect to $B$ (taking into account the fact that $\Gamma\ge1/2$) together with (\ref{Eq-aux1}) leads to our conclusion.
\end{proof}
%
\subsection{Comments\label{com}}
At this stage, several comments are in order:

\paragraph{A simple case}
It is often the case that $\Mcal$ is finite with cardinality $|\Mcal|$ and that we use equal weights 
$\Delta_m=\Delta\le\log(2|\Mcal|)$ for all $\mM$, in which case $\Gamma=
|\Mcal|e^{-\Delta}$ which leads to the following risk bound which only depends on $|\Mcal|$:
\[
\mathbb{E}_s\croch{h^2\paren{s,\widehat{s}_{\widehat{m}_{\TVF}}} } \le2 \paren{\frac{\theta^2+2}{\theta^2}} \inf_{\mM}R\left(\Acal_m,\frac{V-1}{V}n,s\right)+
\frac{4V\log(2e|\Mcal|)}{n(1-2\theta)^2}. 
\]

\paragraph{Modified V-fold}
Unfortunately, there are actually many estimators, like maximum likelihood estimators or T-estimators, that do not satisfy Assumption~(LOSS) and for which the previous risk computations fail. In order to solve this problem, one should think about the initial purpose of VF methods and, more generally, of procedure selection. Starting from the family $\{{\cal A}_m,\,m\in{\cal M}\}$, we want to determine, at least approximately, the best procedure for the problem at hand. But if we design an alternative procedure $\overline{{\cal A}}$ not contained in the initial set, but as good as the best one in the set, we may consider that we have achieved our goal. 

It should be noted at this stage that Assumption~(LOSS) is only used to derive in (\ref{Eq-aux1}) that
\[
\mathbb{E}_s\croch{h^2\paren{s,\widehat{s}_{\widehat{m}}}}\le\frac{1}{V} \sum_{j=1}^V
\mathbb{E}_s\croch{h^2\paren{s,\widehat{s}_{\widehat{m},j}}},
\]
which, in view of Proposition~\ref{prop:MEAN}, holds as soon as $\widehat{s}_{\widehat{m}}= V^{-1}\sum_{j=1}^V\widehat{s}_{\widehat{m},j}$. A natural solution to deal with any family of estimators that do not satisfy Assumption~(LOSS) is therefore as follows. Define the partial estimators $\widehat{s}_{m_j}$ and determine $\widehat{m}_{\rm TVF}$ as before by (\ref{eq:critere TVF}), then define the final TVF-estimator $\widetilde{s}_{\rm TVF}$ by 
\begin{equation}
\widetilde{s}_{\rm TVF}= V^{-1}\sum_{j=1}^V\widehat{s}_{\widehat{m}_{\rm TVF},j}
\label{Eq-modified}
\end{equation}
so that (\ref{Eq-aux1}) is satisfied and the proof proceeds as before; our modified TVF-estimator $\widetilde{s}_{\rm TVF}$ satisfies the conclusion of Theorem~\ref{thm:vfold.unrefined}. 

\paragraph{Extension}
It should be noted that the following analogue of (\ref{Eq-main}) holds (with the same proof)
\[
\mathbb{E}_s\croch{h^2\paren{s,\widehat{s}_{\widehat{m}_{\TVF}}} } \le \inf_{\mM} \set{C_1(\theta,a)R\left(\Acal_m,\frac{V-1}{V}n,s\right)+C_2(\theta,a)
\frac{V(\Delta_m+\log(2\Gamma)+1)}{n} },
\]
if we replace Assumption~(TEST) by the following 
%
\begin{assumption}[\bf TEST']
Let $\theta\in (0,1/2)$ and $a>0$ be given. For all $t$ and $u$ in $\Scal$ and $r=\theta h(t,u)$ there exists some test statistic $T_{t,u,\theta}(\Xbf)$ depending on $t,u,\theta$ and $\Xbf$ with the following properties. The test $\psi_{t,u}$ between $t$ and $u$ defined by
\begin{equation*}
\psi_{t,u}(\Xbf)=\left\{
\begin{array}{lll}
        t \quad \text{if} \quad T_{t,u,\theta}(\Xbf)>z \\
        u\quad \text{if} \quad T_{t,u,\theta}(\Xbf)<z
        \end{array}
\right.,\quad z\in \R,
\end{equation*}
with an arbitrary choice when $T_{t,u,\theta}(\Xbf)=z$, satisfies
\begin{equation*}
\sup_{ \set{P_s \in \Pcal\,| \,h(s,t)\le r}} \mathbb{P}_s\croch{\psi_{t,u}(\Xbf)=u}\le\exp\croch{-nah^{2}(t,u)+z}
\end{equation*}
and
\begin{equation*}
\sup_{ \set{P_s \in \Pcal\,| \,h(s,u)\le r}} \mathbb{P}_s\croch{\psi_{t,u}(\Xbf)=t}\le\exp\croch{-nah^{2}(t,u)-z},
\end{equation*}
where $\mathbb{P}_s$ denotes the probability that gives $\Xbf$ the distribution $P_s^{\otimes n}$.
\end{assumption}
In particular Baraud introduced in \cite{Baraud2011} and for the same purpose of estimator selection the following statistic that relies on a variational formula for the Hellinger affinity. For $r=(t+u)/2$, let
\begin{equation}
T_{t,u}(\Xbf)=\frac{1}{2}\paren{\frac{1}{n}\sum_{i=1}^n \frac{\sqrt{t}(X_i)-\sqrt{u}(X_i)}{\sqrt{r}(X_i)} + \int \paren{\sqrt{t(x)}-\sqrt{u(x)}}\sqrt{r(x)}\,d\mu(x)}.
\label{eq:test_baraud}
\end{equation}
The corresponding test $\psi_{t,u}$ actually satisfies Assumption~(TEST') for small enough constants $\theta$ and $a$. This follows from Baraud (2008, unpublished manuscript). Therefore the test $\psi(t,u)$ derived from Baraud's statistic could be used instead of the tests between balls. Some simulations based on this alternative test will be provided in \secref{SuppMat}.

\section{About the choice of $V$\label{V}}
Let us now come back to the bound (\ref{Eq-main}). It follows from our empirical study in Section~\ref{sub:Influence_theta} below that a good choice of $\theta$ is $1/4$. Therefore assuming, to be specific and for simplicity, that $\theta=1/4$ and that
\begin{equation}
\log(2\Gamma)+1\le3\Delta_m\quad\mbox{for all }\mM,
\label{Eq-gamma}
\end{equation}
(\ref{Eq-main}) becomes
\begin{equation}\label{eq:ccl}
\mathbb{E}_s\croch{h^2\paren{s,\widehat{s}_{\widehat{m}_{\TVF}}} } \le66 \inf_{\mM} \set{R\left(\Acal_m,\frac{V-1}{V}n,s\right)+\frac{V\Delta_m}{n} }. 
\end{equation}
Although this risk bound is certainly far from optimal in view of the large constant 66 and our extended simulations show that the actual risk is indeed substantially smaller, it is nevertheless already enlightening. To see it, let us begin with the simple case of regular histograms.	
%
\subsection{Regular histograms}\label{sect-H}
Let us analyze the problem of estimating an unknown density $s$ with respect to the Lebesgue measure on $[0,1]$ from $n$ i.i.d.\ observations with density $s$. We consider, for each positive integer $m$, the histogram estimator $\widehat{s}_m$ based on the partition $\Ical_m$ of $[0,1]$ into $m$ intervals of equal length $m^{-1}$. It is known from  \cite[Theorem 1]{Birge-Rozenholc2006} that the risk at $s$ of the histogram $\widehat{s}_m$ built from $n$ i.i.d.\ observations is bounded by
\begin{equation}
\mathbb{E}_s\croch{ h^2\paren{s,\widehat{s}_m}} \le h^2\paren{s,\overline{s}_m}+\frac{m-1}{2n},
\label{Eq-histo1}
\end{equation}
where $\overline{s}_m$ is the $\Lbb_2$-projection of $s$ onto the $m$-dimensional linear space of piecewise constant functions on the partition $\Ical_m$. It is also shown in this theorem  that this bound is asymptotically optimal, up to a factor 4, since the asymptotic risk (when $n$ goes to infinity) is of the form
\begin{equation}
\mathbb{E}_s\croch{ h^2\paren{s,\widehat{s}_m}}=h^2\paren{s,\overline{s}_m}+\frac{m-1}{8n}\left(\strut1+o(1)\right).
\label{Eq-histo2}
\end{equation}
In view of (\ref{Eq-histo2}), the bound in (\ref{Eq-histo1}) can be considered as optimal, up to
a constant factor and it follows from (\ref{Eq-histo1}) that
\begin{equation}
R\left(\Acal_m,\frac{V-1}{V}n,s\right)\leq h^2\paren{s,\overline{s}_m}+\frac{(m-1)V}{2n(V-1)}=h^2\paren{s,\overline{s}_m}+\frac{m-1}{2n}+\frac{m-1}{2n(V-1)} 
\label{Eq-histo3}
\end{equation}
and that
\begin{equation}
\inf_{\mM} \mathbb{E}_s\croch{ h^2\paren{s,\widehat{s}_m}} \le h^2\paren{s,\overline{s}_{m^*}}+\frac{(m^*-1)}{2n}=\inf_{\mM} \set{h^2\paren{s,\overline{s}_m}+\frac{(m-1)}{2n}},
\label{Eq-histo6}
\end{equation}
where this last bound can be considered as a benchmark for the risk of any selection 
procedure applied to our family of histograms. Since the Hellinger distance is bounded by 1, it clearly appears that one should restrict to values of $m$ that are not larger than $2n$. We shall therefore now assume that $\Mcal=\{1,2,\ldots,2n\}$.

Applying (\ref{Eq-histo3}) to (\ref{eq:ccl}), we get
\begin{eqnarray}
\frac{1}{66}\mathbb{E}_s\croch{h^2\paren{s,\widehat{s}_{\widehat{m}_{\TVF}}} }
&\le&\inf_{\mM} \set{R\left(\Acal_m,\frac{V-1}{V}n,s\right)+\frac{V\Delta_m}{n}}\label{B0}\\&\le& \inf_{\mM} 
 \set{\left(h^2\paren{s,\overline{s}_m}+\frac{m-1}{2n}\right) +\left(\frac{m-1}{2n(V-1)}
+\frac{V\Delta_m}{n}\right)}\quad\label{B1}\\&\le&\left[h^2\paren{s,\overline{s}_{m^*}}+\frac{(m^*-1)}{2n}\right]+\left[\frac{m^*-1}{2n(V-1)}+\frac{V\Delta_{m^*}}{n}\right]
\label{B2},
\end{eqnarray}
with $m^*$ defined by (\ref{Eq-histo6}).
We see from (\ref{B1}) that, up to the multiplicative constant 66, we have to optimize with 
respect to $m$ a bound for the risk of $\widehat{s}_m$ plus a residual term which depends 
in a non-monotonous way of $V$. The bound (\ref{B2}) shows that, up to a constant factor,
we actually recover our benchmark (\ref{Eq-histo6}) plus an error term which writes
\[
g(V-1)\quad\mbox{with}\quad g(x)=\frac{1}{n}\left(\frac{m^*-1}{2x}+x\Delta_{m^*}\right)+\frac{\Delta_{m^*}}{n}.
\]
Clearly, $g(x)$ is minimum for $x=x_0=\sqrt{(m^*-1)/(2\Delta_{m^*})}$. It follows that the optimal value of $V$ is two if $m^*-1\le2\Delta_{m^*}$. This occurs in particular if $m^*=1$, for instance when $P_s$ is the uniform distribution on $[0,1]$ or close enough to it. It also occurs if $\Delta_{m}\ge(m-1)/2$ for all $m\ge2$. 

Let us now consider the situation for which $m^*-1>2\Delta_{m^*}$ so that $x_0>1$ and the optimal value of $V$ belongs to $(x_0-1,x_0+1)$. If $(m-1)/\Delta_{m}$ is an increasing function of $m$, the optimal value of $V$ will be a non-decreasing function of $m^*$ which, as $m^*$ does, depends on the true unknown value of $s$, large values of $m^*$ leading to large values for $V$ and vice-versa. For instance, the choice of equal weights, $\Delta_m=\log2n$ for $m\in\Mcal$ leads to $\Gamma=1$ which satisfies 
(\ref{Eq-gamma}) and to an optimal $V$ of order $\sqrt{(m^*-1)/(2\log2n)}$. But this choice of $\Delta_m$ is certainly not optimal in view of (\ref{B0}). A better one would be $\Delta_m=(1/3)+2\log m$ which also satisfies (\ref{Eq-gamma}) but improves (\ref{B0}) substantially. Then the optimal value of $V$ is of order $\sqrt{(m^*-1)/((2/3)+4\log m^*)}$, still depending on the true unknown $s$. Only larger values of $\Delta_m$ of the form $\Delta_m=a(m-1)$ for $m\ge2$ that deteriorate the bound (\ref{B0}) and therefore should not be recommended lead to an optimal value of $V$ which is independent of $m^*$, hence of $s$. 
%
\subsection{The typical situation}\label{sect-V2}
A risk bound of the form (\ref{Eq-histo1}) is actually not specific of histograms but actually rather typical. There are many procedures for which the risk, for a convenient choice of the index $m$ and of the set $\Mcal\subset\Bbb{R}$ can be bounded in the following way:
\begin{equation}
\mathbb{E}_s\croch{ h^2\paren{s,\widehat{s}_m}} \le H(s,m)+Cmn^{-1},
\label{Eq-RB}
\end{equation}
where $H$ is a nonincreasing function of $m$, leading to an optimal choice $m^*$ for $m$ (with respect to this bound which we take as a benchmark for the risk) given by 
\[
m^*=\argmin_{\mM} \set{H(s,m)+Cmn^{-1}}.
\]
It then follows from (\ref{eq:ccl}) that we get an analogue of (\ref{B2}), namely
\begin{eqnarray*}
\frac{1}{66}\mathbb{E}_s\croch{h^2\paren{s,\widehat{s}_{\widehat{m}_{\TVF}}} }&\le&\inf_{\mM} 
\set{H(s,m)+\frac{CmV}{n(V-1)}+\frac{V\Delta_m}{n}}\\&\le&\left[H(s,m^*)+\frac{Cm^*}{n}\right]
+\frac{1}{n}\left[\frac{Cm^*}{V-1}+(V-1)\Delta_{m^*}\right]+\frac{\Delta_{m^*}}{n}
\end{eqnarray*}
and we see that the choice of $V$ is driven, as in the case of regular histograms, by the quantity
\begin{equation}
(V-1)^{-1}Cm^*+(V-1)\Delta_{m^*}.
\label{Eq-V}
\end{equation}
The same arguments as before show that the optimal choice of $V$ then depends on the ratio 
$m^*/\Delta_{m^*}$ and therefore on $s$ in many situations. This dependence of the optimal value of $V$ with respect to the true density $s$ will actually be confirmed by our simulations below. A density which is difficult to estimate by a histogram with a few bins will lead to a large value of $m^*$ hence a large optimal $V$ while a simple density, for which $m^*$ is rather small, is better estimated by a V-fold with a small $V$.
In the case of a finite set $\Mcal$, which is the practical one, and of equal weights, which is the simplest but suboptimal choice, the optimal $V$ varies like $m^*$.

\subsection{Kernel estimators\label{sec:ker}}
We consider here estimation of an unknown density $s$ by a kernel estimator $\widehat{s}_w$ using a nonnegative kernel $K$ and a positive bandwidth $w$ which means that
\begin{equation}
\widehat{s}_w(x)=\sum_{i=1}^nK_w(x-X_i)\quad\mbox{with}\quad K_w(y)=w^{-1}K\left(w^{-1}y\right).
\label{Eq-kernest}
\end{equation}
Although there are many papers which study the performance of kernel estimators, in particular their risk with respect to $\Bbb{L}_p$-type losses, we were unable to find a result about their non-asymptotic risk with respect to the squared Hellinger loss. This is why we provide one below, the proof of which is deferred to Section~\ref{P}.
%
\begin{theorem}\label{T-kernel}
Let $s$ be a density on the real line which is supported on an interval of length $2L$ and such that 
$\sqrt{s}$ has an $\Bbb{L}_2$-modulus of continuity  
\begin{equation}
\omega_{2}\left(\sqrt{s},\eta\right)=\sup_{\left\vert z\right\vert \le\eta}
\left\Vert\sqrt{s}\left(\cdot+z\right)-\sqrt{s}\right\Vert=\sqrt{2}\sup_{\left\vert z\right\vert \le\eta}
h\left(s(\cdot+z),s\strut\right).
\label{Eq-modcon}
\end{equation}
Let $\phi$ be a nondecreasing and concave function on $[0,+\infty)$ with $\phi(0)=0$ and 
$\omega_{2}\left(\sqrt{s},\eta\right)\le\phi(\eta)$ for $\eta\ge0$. Assume moreover that the
kernel $K$ is bounded with $\int x^2K(x)\,dx<+\infty$ and that it is ultimately monotone around 
$-\infty$ and $+\infty$. Then the kernel estimator $\widehat{s}_w$ given by (\ref{Eq-kernest}) 
satisfies
\begin{equation}
\Bbb{E}_s\left[h^2\left(\widehat{s}_w,s\right)\right]\le2\left[\int_{\mathbb{R}}\left(1\vee x^{2}\right)K(x)\,dx\right]\phi^{2}(w)+\frac{2L\|K\|_\infty}{nw}+\frac{C(K)}{n},
\label{Eq-kernel risk}
\end{equation}
where the constant $C$ only depends on the kernel $K$ and is equal to 1 when $K$ is unimodal.
\end{theorem}
If we restrict to densities $s$ with a known compact support, this bound takes the form (\ref{Eq-RB}) with the choice $m=w^{-1}+1$. A ``classical'' smoothness assumption on $\sqrt{s}$ corresponds to the choice
$\phi(\eta)=M\eta^{\alpha}$, for some exponent $\alpha\in(0,1]$. In this case the smallest quantity $M$ 
such that $\omega_{2}\left(\sqrt{s},\eta\right)\leq\phi(\eta)$ holds true is merely the Besov
semi-norm of $\sqrt{s}$ in the Besov space $B_{2,\infty}^{\alpha}$. In such a case,
we see that the optimal value of $\eta$ is of order $n^{-1/(2\alpha+1)}$, leading to a risk bound of order
$n^{-2\alpha/(2\alpha+1)}$. This is completely analogous to what we get for the squared $\Bbb{L}_2$-risk, apart from the fact that for Hellinger we put the smoothness assumption on $\sqrt{s}$ instead of $s$.

\subsection{Handling arbitrary estimators\label{}}
The previous construction of TVF-estimators is only valid for genuine preliminary density estimators
$\widehat{s}_m$, that is such that $\widehat{s}_m(x)\ge0$ for all $x\in{\cal X}$ and 
$\int\widehat{s}_m(x)\,d\mu(x)=1$, but this is definitely not the case for all classical estimators. For instance additive estimators given by (\ref{eq:linear estim}) do not satisfy these requirements when the function $\Kcal$ may take negative values. This actually happens for projection estimators derived from wavelet expansions or kernel estimators based on kernels that take negative values. Not only TVF-estimators cannot be built from
preliminary estimators that take negative values but the Hellinger distance cannot be defined for such estimators since it involves the square roots of the densities. There is actually a simple and reasonable solution to this problem which is to transform any function $t$ such that $\int_{t>0}t\,d\mu>0$ into a probability density $\pi(t)$ with respect to $\mu$ using the following operator $\pi$:
\begin{equation}
\pi(t)=\frac{t\vee0}{\int(t(x)\vee0)\,d\mu(x)}.
\label{Eq-pi}
\end{equation}
It is clear that for any probability density $s$, $|s(x)-(t(x)\vee0)|\le|s(x)-t(x)|$ so that $t\vee0$ is closer
from $s$ than $t$ for any reasonable distance, including all $\Bbb{L}_p$-distances. Moreover the following lemma shows that $h(s,\pi(t))\le\left\|\sqrt{s}-\sqrt{t\vee0}\right\|$ which justifies the use of the transformation $\pi$ when dealing with the Hellinger distance. 
%
\begin{lemma}\label{L-L_2proj}
Let $f,g$ be two nonegative elements of $\Bbb{L}_2(\mu)$ with $\|f\|=1$ and $\|g\|>0$. Let $\overline{g}=
g/\|g\|$ so that $\|\overline{g}\|=1$. Then
\[
\|f-\overline{g}\|^2\le\frac{4\|f-g\|^2}{4-\|f-\overline{g}\|^2}\le2\|f-g\|^2.
\]
If $s$ is a density with respect to $\mu$, $g$ a nonegative element of $\Bbb{L}_2(\mu)$ with positive norm and $u=(g/\|g\|)^2$, then $u$ is also a density with respect to $\mu$ and
\[
h^2(s,u)\le1-\sqrt{1-\left(\left\|\sqrt{s}-g\right\|^2\wedge1\right)}\le\left\|\sqrt{s}-g\right\|^2\wedge1.
\]
If, in particular, $t$ is an arbitrary element of $\Bbb{L}_1(\mu)$ such that $\int(t\vee0)\,d\mu>0$, then
\[
h^2(s,\pi(t))\le1-\sqrt{1-\left(\left\|\sqrt{s}-\sqrt{t\vee0}\right\|^2\wedge1\right)}\le
\left\|\sqrt{s}-\sqrt{t\vee0}\right\|^2\wedge1.
\]
\end{lemma}
\begin{proof}
Let $\|g\|=\lambda$ so that $g=\lambda\overline{g}$ and let $\rho=\scal{f}{\overline{g}}\in[0,1]$. Then
\[
\|f-g\|^2=1+\lambda^2-2\lambda\rho\qquad\mbox{and}\qquad\|f-\overline{g}\|^2=2(1-\rho)\le2.
\]
It follows that, for a given value of $\rho$, the minimum value of $\|f-g\|^2$ is obtained for
$\lambda=\rho$ and equal to $1-\rho^2$ which implies that
\begin{equation}
\left(\frac{\|f-\overline{g}\|}{\|f-g\|}\right)^2\le\frac{2}{1+\rho}=\frac{4}{4-\|f-\overline{g}\|^2}\le2.
\label{Eq-proj2}
\end{equation}
If $f=\sqrt{s}$, then $\rho=\rho(s,u)=1-h^2(s,u)$ and 
$\|f-\overline{g}\|^2=2h^2(s,u)$, so that (\ref{Eq-proj2}) becomes
\[
h^2(s,u)\le\frac{\|f-g\|^2}{1+\rho(s,u)}=\frac{\|f-g\|^2}{2-h^2(s,u)}
\]
and, since $h(s,u)\le1$, it also follows from elementary calculus that
\[
h^2(s,u)\le1-\sqrt{1-\left(\|f-g\|^2\wedge1\right)}\le\|f-g\|^2\wedge1.
\]
The last inequality is just the case of $g=\sqrt{t\vee0}$.
\end{proof}
Using the transformation $\pi$ amounts to replace the initial family $\{{\cal A}_m,\,m\in{\cal M}\}$ by a new one $\{{\cal A}'_m,\,m\in{\cal M}\}$ via the tranformation ${\cal A}'_m(\mathbf{Y}_k)=\pi\left(\strut{\cal A}_m(\mathbf{Y}_k)\right)$ which results in  procedures that now make sense for the Hellinger loss.
Unfortunately, this transformation does not preserve the linearity so that if we apply this recipe to projection or kernel estimators, we cannot know whether the transformed estimators satisfy Assumption~(LOSS). Nevertheless, as we have seen in Section~\ref{com}, we may change the definition of TVF-estimators to (\ref{Eq-modified}) in order to solve this problem. 

Starting from a family of estimators that are not probability densities, we merely begin
with a preliminary application of the transformation $\pi$, as given by (\ref{Eq-pi}),
and then define our modified TVF-estimator via (\ref{Eq-modified}) so that Theorem~\ref{thm:vfold.unrefined} applies to the family of procedures $\{{\cal A}'_m,\,m\in{\cal M}\}$.
\section{Empirical study}\label{sec:Empirical study}
The theoretical bounds that we have derived, for instance (\ref{eq:ccl}), are quite pessimistic because of the large constants that are present in our risk bounds. It is therefore crucial to know whether such large values are only artifacts or really enter the risk. In order to check the real
quality of our selection procedure and evaluate the influence of the various parameters involved in it, we performed an extensive set of simulations the results of which are summarized below.

\subsection{Simulation protocol}\label{sub:Simulation protocol}
We studied the performances of the TVF procedure on 18 out of the 28 densities described in the \textit{benchden} \footnote{Available on the CRAN \url{http://cran.r-project.org}.} R-package \cite{Mildenberger-Weinert2012} which provides a full implementation of the distributions introduced in \cite{Berlinet-Devroye1994} as benchmarks for nonparametric density estimation. We only show our simulations for the eleven densities in the subset $\Lcal=\{s_i,~ i=1,2,3,4,5,7,12,13,22,23,24\}$ (where the indices refer to the list of benchden) the graphs of which are shown in  \figref{densities}, except for the uniform density $s_1$ on $[0,1]$.
\begin{figure}[h]
\[ \includegraphics[height=3cm]{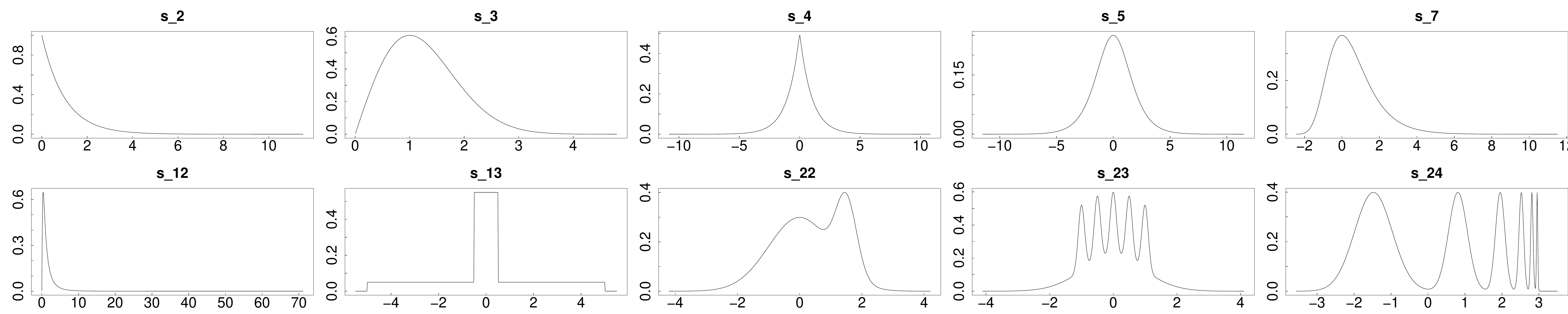} \]
\caption{\label{fig:densities} \small Graphs of all densities mentioned in the paper apart from the uniform.}
\end{figure}
For a given loss $\ell=h^2,d_1$ or $d_2^2$ (respectively the squared Hellinger, $\Lbb_1$- and squared 
$\Lbb_2$-losses), we decided to evaluate the accuracy of some estimator $\widetilde s=\widehat{s}_{\widehat{m}}$ by empirically estimating its risk $R(\widetilde{s},s,\ell)=\mathbb{E}_s\croch{\ell(s,\widetilde{s})}$. To do so, we generated $1000$ pseudo-random samples $\Xbf^i=\{X_1^i,\ldots,X_n^i\}$, $1\leq i \leq 1000$, of size $n$ and density $s$ and approximated $R(\widetilde{s},s,\ell)$ by its empirical version:
\[
\overline{R}_n\paren{\widetilde{s},s,\ell}=\frac{1}{1000} \sum_{i=1}^{1000} \ell\paren{s,\widetilde{s}(\Xbf^i)}.
\]

As in \cite{Magalhaes-Rozenholc2014}, we considered several families of estimators. We present here our simulations for the well-known problems of bandwidth selection for kernel estimators with a Gaussian kernel and the choice of the bin number for regular histograms. We therefore introduce the following families of estimators.
\begin{itemize}
\item $\Fcal_{\rm R}$ is the set of regular histograms with bin number varying from 1 to 
$\lceil n/\log n\rceil$ as described in \cite{Birge-Rozenholc2006},

\item $\Fcal_{\rm K}$ is the set of Gaussian kernel estimators with bandwidths $w_m$ of the form  
\[		
w_m=\frac{1}{n\log n}\paren{1+\frac{1.5}{\log n}}^m,\qquad\text{for }m=1,\ldots,\paren{\log n}^2,
\]

\item $\Fcal_{\rm KR}=\Fcal_{\rm K} \cup \Fcal_{\rm R} $.
\end{itemize}
Besides the classical VF methods, we considered two alternative procedures that are known to perform well in practice in order to have an idea of the performance of the T-V-fold as compared to some especially calibrated methods. When studying the problem of bandwidth selection, we compared the TVF with the unbiased cross-validation selector, implemented in the {\it density} generic function available in R, which provides an estimator which does not belong to the set $\{\widehat{s}_m,\,\mM\}$. When dealing with the bin number choice we implemented the penalization procedure of Birg\'e and Rozenholc (described in \cite{Birge-Rozenholc2006}) which selects a regular histogram in $\Fcal_{\rm R}$. These two competitors will be denoted ``UCV'' and ``BR'' respectively in our study. To implement the TVF and process our simulations we used an algorithm which is described in \secref{Algorithm} with the tests defined in \eqref{eq:test_birge} and constant weights $\Delta_m=\Delta=0$ for all $\mM$. 

We made thousands of simulations (varying the sample size $n$, the density, the family of estimators, the number $V$ of splits in our V-fold procedures, etc.) but since the results we found were very similar in all situations, we only show the conclusion for $n=500$ and $V=2,5,10$ and 20.

\subsection{The influence of the parameter $\theta$}\label{sub:Influence_theta}
 As in \cite[Section 5.1]{Magalhaes-Rozenholc2014}, we have studied the influence on the performance of the TVF procedure of the parameter $\theta$ that is used in the definition of the test statistic \eqref{eq:test_birge}. The parameter influences the performance of the tests $\psi_{t,u}$ as shown by (\ref{eq:robustness1}) and (\ref{eq:robustness2}) and therefore the whole procedure. Since on the one hand $\theta=0$ corresponds to the KLVF and on the other hand 
 $\theta$ must be less than 1/2, we made comparisons between the versions of $\widetilde s_{\rm TVF}$ deduced from the tests with $\theta\in \Theta=\{1/16,1/8,1/4,3/8,7/16\}$. For the sake of clarity and to emphasize the stability of the behavior of the procedure in terms of risk, we present for each $V$ the ratio 
\begin{equation}
\inf_{s\in\Lcal} \set{\inf_{\theta\in \Theta} \overline{R}_n\paren{\widehat{s}_{\widehat{m}(\theta)},s,h^2} \Big/ \sup_{\theta\in \Theta} \overline{R}_n\paren{\widehat{s}_{\widehat{m}(\theta)},s,h^2} }, 
\label{Eq-ratio}
\end{equation}
which gives the largest difference in terms of risk among the densities in $\Lcal$. The closer the ratio to 1, the more stable the procedure with respect to the variations of $\theta$.
\begin{table}[h]
\centering
\small
\begin{tabular}{|l|r|r|r|r|}\hline
family&$V=2$&$V=5$&$V=10$&$V=20$\\\hline
$\Fcal_{\rm R}$&92,95&94,87&96,39&96,96 \\
\hline
\hline
$\Fcal_{\rm K}$&91,31&92,94&94,79&96,44\\
\hline
\hline
$\Fcal_{\rm KR}$&87,81&94,36&97,48&95,15\\
\hline
\end{tabular}
\caption{\label{tab:theta} \small 100 times the ratio (\ref{Eq-ratio}) for $n=500$ and families 
$\Fcal_{\rm R}$, $\Fcal_{\rm K}$ and $\Fcal_{\rm KR}$.}
\end{table}
We may conclude from this picture that $\theta$ has little influence on the quality of the resulting estimator for families $\Fcal_{\rm K}$ and $\Fcal_{\rm R}$, even if we did observe that $\theta=1/16$ is in general slightly worse than the other values (in particular for the family $\Fcal_{\rm R}$). 
Considering family $\Fcal_{\rm KR}$, we have observed that there might be some noticeable difference for $V=2$ for one specific density. 
Nevertheless no clear conclusion can be derived from our simulations as the best value of $\theta$ varies with the setting.
Finally, it appears that the choice $\theta=1/4$ is always satisfactory and should be recommended.
\subsection{About the choice of $V$}\label{sub:Influence_V}

The main question when considering VF type procedures is maybe ``which V is optimal?''  or, more generally, ``what is the influence of $V$ on the quality of the VF procedure?''.  According to our theoretical study in Section~\ref{V} the optimal value of $V$ depends on the optimal value $m^*$ of $m$. In the case of equal weights the best $V$ appears to be an increasing function of $m^*$. In the case of histograms, if the best one has many bins, one should take a large value of $V$ and the same would hold for a kernel estimator with a small bandwidth. To understand what actually happens in practice, we study here how the risk of the chosen estimator behaves when $V$ varies. 

Since $\theta$ has little influence, we made all the simulations with $\theta=1/4$. We also implemented the calibrated procedures BR and UCV described in \subref{Simulation protocol} in order to have a benchmark for the risk for the families $\Fcal_{\rm R}$ and $\Fcal_{\rm K}$ respectively.
\begin{table}[h]
\centering
\small
\begin{tabular}{|l|c|r|r|r|r|r|r|r|r|r|r|r|}
\hline
family&$V$&$s_1$&$s_2$&$s_3$&$s_4$&$s_5$&$s_7$&$s_{12}$&$s_{13}$&$s_{22}$&$s_{23}$&$s_{24}$\\\hline
\multicolumn{1}{|c|}{\multirow{5}{*}{$\Fcal_{\rm R}$}}&$2$&\textbf{2,9}&10,4&9,29&13,8&10,9&11,4	&17,9&14,5	&10,5	&20,8	&27,5\\
\multicolumn{1}{|c|}{}&$5$&4,31	&9,9 &8,75&12,7&10  &10,6&17,3&\textbf{13,5}&9,56&18,4&25,2\\
\multicolumn{1}{|c|}{}&$10$	&6,18&9,81&8,64	&12,3&9,77&10,6	&\textbf{17,2}&13,7&9,51	&\textbf{17,8}&\textbf{24,8}\\
\multicolumn{1}{|c|}{}&$20$	&9,39&\textbf{9,65}&\textbf{8,54}&\textbf{12,2}&\textbf{9,59}&\textbf{10,4}&17,3&14,1&\textbf{9,28}	&17,9&\textbf{24,8}\\\cline{2-13}
\multicolumn{1}{|c|}{}&BR&2,20&9,94&9,27&12,98&10,53&11,14&17,85&14,63&10,37&17,98&25,15\\
\hline
\hline
\multicolumn{1}{|c|}{\multirow{5}{*}{$\Fcal_{\rm K}$}}&$2$	&15,4&29,9&5,67	&5,1 &\textbf{3,56}&4,26	&28,5&20 &3,96	&10,6	&18,1\\
\multicolumn{1}{|c|}{}&$5$	&12,7&25,5&5,06&\textbf{4,95}&3,61	&\textbf{3,98}	&23,4	&18,1	&\textbf{3,86}	&9,28	&16,2\\
\multicolumn{1}{|c|}{}&$10$	&12,4&24,3&\textbf{4,94}&5,01&3,96	&4,04	&21,8	&17,7	&3,91	&9,08	&15,8\\
\multicolumn{1}{|c|}{}&$20$	&\textbf{12,2}&\textbf{23,5}&4,97	&5,41	&4,9 	&4,27	&\textbf{20,9}	&\textbf{17,6}	&4,11	&\textbf{9,05}	&\textbf{15,7}\\\cline{2-13}
\multicolumn{1}{|c|}{}&UCV&15,86&22,20&5,57&6,16&3,74&4,10&18,80&17,16&3,88&9,52&15,91\\
\hline
\hline
\multicolumn{1}{|c|}{\multirow{4}{*}{$\Fcal_{\rm KR}$}}&$2$	&\textbf{2,88}&10,4&8,32&6,35&5,81&6,57&18,5&14,4	&7,3 	&12,8	&20  \\
\multicolumn{1}{|c|}{}&$5$	&4  &9,91&7,86	&\textbf{5,64}&\textbf{5,11}&\textbf{6,06}&17,7&\textbf{13,2}&\textbf{5,76}&9,66	&16,7\\
\multicolumn{1}{|c|}{}&$10$	&4,34&9,95&7,66	&\textbf{5,64}&5,4&6,18&17,6&13,7&5,82&9,12	&16  \\
\multicolumn{1}{|c|}{}&$20$	&4,34&\textbf{9,86}&\textbf{7,49}&5,91&5,81&6,5 &\textbf{17,5}&14,5&5,88&\textbf{9,08}&\textbf{15,7}\\
\hline
\end{tabular}  
\caption{\label{tab:influence_V_Birge} \small $10^3$ times the Hellinger risks of the TVF procedure.}
\end{table}

The empirical results summarized in \tabref{influence_V_Birge} actually confirm what we derived from (\ref{Eq-V}). The quality of the estimation increases with $V$ when
the true density is difficult to estimate which corresponds to an optimal estimator $\widehat{s}_{m^*}$
with a large value of $m^*$ in (\ref{Eq-V}). For a simple density like the uniform $s_1$ which is better estimated by an histogram with few bins, the best choice of $V$ is 2 for the families $\Fcal_{\rm R}$ and $\Fcal_{\rm KR}$ which include histograms. On the contrary, when dealing with the family $\Fcal_{\rm K}$ for which $s_1$ is not easy to estimate, we need to use a larger value of $V$. A similar situation occurs with densities $s_4$, $s_5$, $s_7$ and $s_{22}$ which appears to be easily estimated by a kernel estimator with a large bandwidth but poorly by histograms. It seems that, apart from the exceptional situation of $s_1$, the best value of $V$ is not $2$ and the most significant gain appears between $V=2$ and $V=5$, then the quality sometimes keeps improving from $V=5$ to $V=20$, but with very little difference between $V=10$ and $V=20$.

Interestingly, we also observe that when using the mixed collection $\Fcal_{\rm KR}$ the TVF procedure shows a good adaptation behaviour since it selects the best family in all settings. For instance for $s_5$ it chooses a kernel estimator since these are better than histograms for estimating it, whereas it selects an histogram for $s_2$ for the opposite reason. 

The numerical complexity of the TVF procedure is quite important in practice and increases with $V$ so that large values of $V$ should be avoided because they lead to a much larger computation time. In particular the Leave-one-out ($V=n$) should be excluded since it is typically impossible to compute it in a reasonable amount of time. 
Of course, since the optimal value of $V$, as we have seen, depends on unknown properties of the procedures with respect to the true density it is not possible to practically define an optimal choice of $V$. Nevertheless our empirical study suggests that a good compromise, which leads to both a reasonable computation time and a good performance (apart from some exceptional situations like the estimation of the uniform by histograms), is $V=5$. We would therefore recommend the user to process the TVF procedure with this value.
\subsection{Comparison with others VF procedures}
The goal of this section is to compare our TVF procedure with other general VF procedures namely LSVF and KLVF, which do not depend on the family of estimators from which we estimate $s$. In order to compare two VF procedures $\tilde t_1$ and $\tilde t_2$, we consider the $\log_2$-ratio of their empirical risk, 
\begin{equation*}
\overline{W}_s\paren{\tilde t_1,\tilde t_2,\ell}= \log_2\frac{\overline{R}_n\paren{\tilde t_1,s,\ell}}{\overline{R}_n\paren{\tilde t_2,s,\ell}} .
\end{equation*}
Thus $\overline{W}_s\paren{\tilde t_1,\tilde t_2,\ell}=c$ means that $\overline{R}_n\paren{\tilde t_1,s,\ell}
=2^{c}\times\overline{R}_n\paren{\tilde t_2,s,\ell}$. Hence, for a given density $s$, $\tilde t_2$ is a better estimator than $\tilde t_1$ if $c>0$. In our empirical study, a selection procedure $\tilde t_2$ is thus considered better than $\tilde t_1$ in terms of risk for a given loss function $\ell$ if the values of $\overline{W}_s(\tilde t_1,\tilde t_2,\ell)$ are positive when the density $s$ varies in $\Lcal$. 

Rather than presenting exhaustive results, that is the evaluation of $\overline{W}_s$ for all densities $s$ in $\Lcal$, different loss functions, various observations numbers $n$ and different choices of $V$, we shall illustrate the results of our simulations by showing boxplots of $\{\overline W_s(\tilde t_1,\tilde t_2,\ell),~s\in \Lcal\}$ with the discriminating value zero emphasized in red. We actually observed similar results and behaviours for all losses and all sample sizes and therefore present here only the results for $\ell=h^2$ and $n=500$ for the sake of simplicity. \figref{TVF_Birge vs contrastVF (hellinger)} is built with $\tilde t_1=\widehat{s}_{\widehat{m}_{{\rm LSVF}}}$ (upper line) or $\widehat{s}_{\widehat{m}_{{\rm KLVF}}}$ (bottom line) and $\tilde t_2=\widehat{s}_{\widehat{m}_{\TVF}}$ with $\theta=1/4$.
\begin{figure}[h!]
\[ \includegraphics[height=4cm]{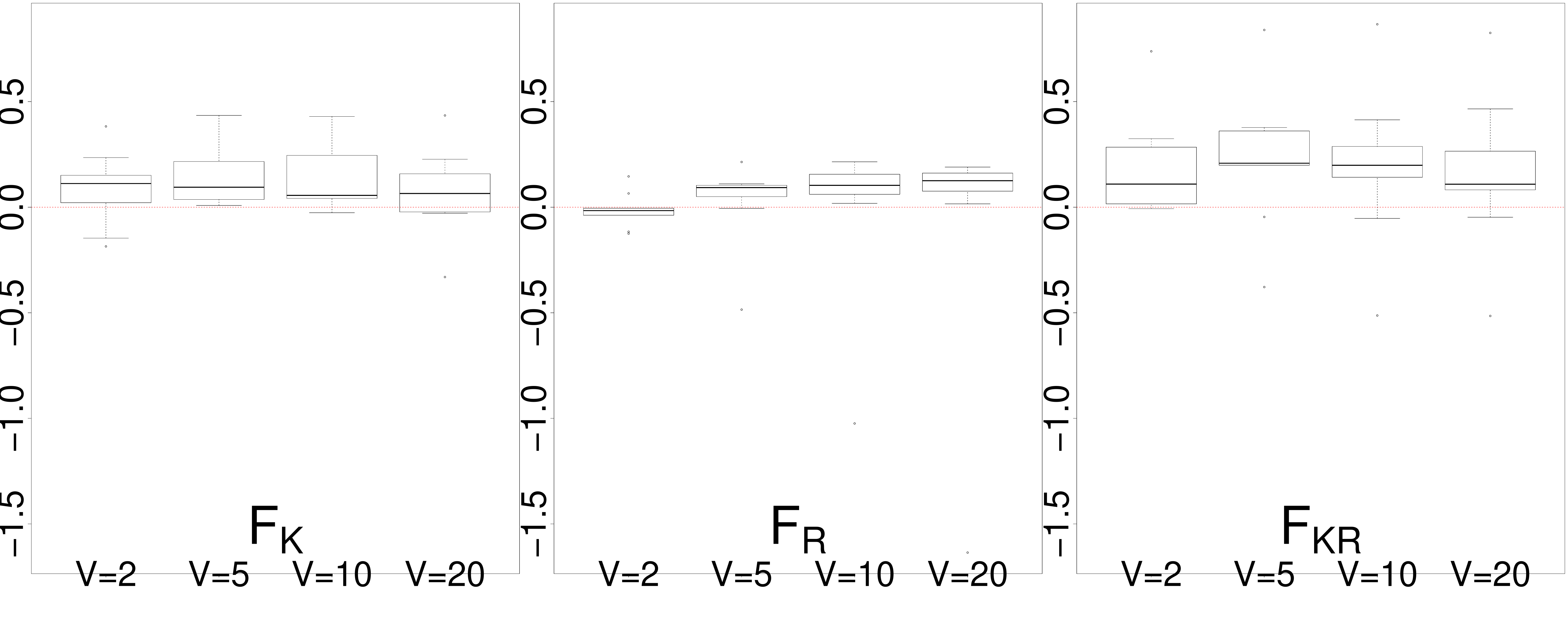} \]
\[ \includegraphics[height=4cm]{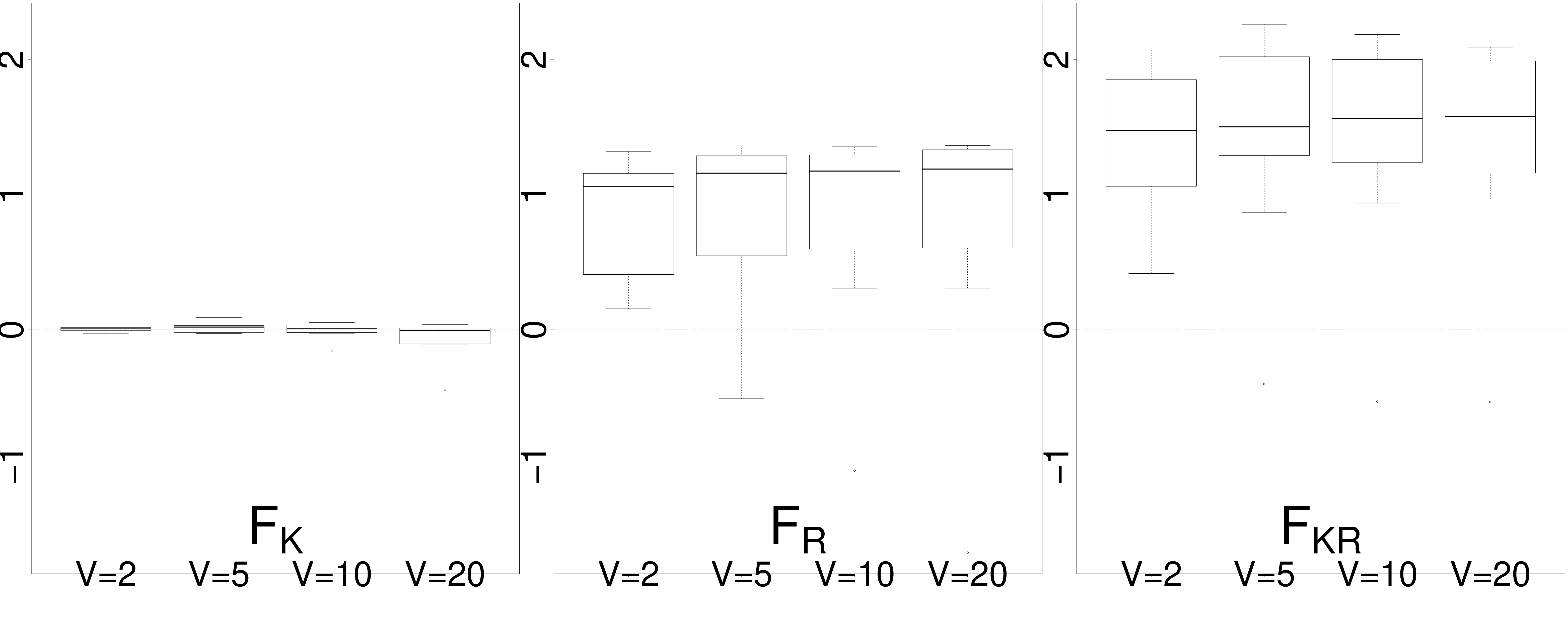} \]
\caption{\label{fig:TVF_Birge vs contrastVF (hellinger)} \small From left to right, the boxplots $\overline{W}_s(\tilde s,\widehat{s}_{\widehat{m}_{\TVF}},h^2)$ using families $\Fcal_{\rm K},\Fcal_{\rm R},\Fcal_{\rm KR}$ (up for $\tilde s=\widehat{s}_{\widehat{m}_{{\rm LSVF}}}$, down for $\tilde s=\widehat{s}_{\widehat{m}_{{\rm KLVF}}}$). Each subfigure shows the boxplots for $V=2,5,10$ and 20. The horizontal red dotted line indicates the reference value 0.}
\end{figure}

In nearly all cases, the median and most of the distribution are positive, which means that the TVF outperforms LSVF (with an average gain of about 20\% for the three families of estimators $\Fcal_{\rm K}$, 
$\Fcal_{\rm R}$ and $\Fcal_{\rm KR}$) and KLVF as well. For the collection $\Fcal_{\rm K}$ we observe that the empirical risks of TVF and KLVF are similar with boxplots of $\overline{W}_s(\widehat{s}_{\widehat{m}_{{\rm KLVF}}},\widehat{s}_{\widehat{m}_{\TVF}},h^2)$ highly concentrated around zero. But
there is a huge difference between TVF and KLVF procedures for families $\Fcal_{\rm R}$ and  $\Fcal_{\rm KR}$  (average gain of about 100\% and 180\% respectively). 
For the uniform density estimated with regular histograms, the estimator derived from our procedure is worse since we found, for both classical VF, $\overline{W}_{s_1}(\tilde s,\widehat{s}_{\widehat{m}_{\TVF}},h^2)<0$ (with an increasing difference with $V$ for $\Fcal_{\rm R}$).
 Finally, let us notice that the difference between TVF and classical VF does not change much with $V$.

\section{Our computational algorithm}\label{sec:Algorithm}
For the practical computation of the TVF as well as any other VF procedure, we assume that $\Mcal$ is finite with cardinality $M$.

Let us compare the complexity of a classical $V$-fold method with ours. Since for every VF method the 
construction of all partial estimators $(\widehat{s}_{m,j})_{1\leq j \leq V, 1\leq m \leq M}$ is required,
we only have to focus on the ``validation part'' which requires to compute all quantities $\Dcal^2_j(m)$ for 
$1\le j \le V$ and $\mM$ and therefore to perform all tests $\psi_{l,m}(\Xbf_{j})$ for $1\le j \le V$ 
and $l,m\in\Mcal$ with $l\ne m$. This means performing $V\times M\times (M-1)/2$ tests leading to a 
computational cost of order $O(V\times M^2)$ that can be prohibitive as compared to the one of 
either LSVF or KLVF which have a maximum complexity of order $O(V\times M)$ (since in this case 
no more than $M$ calculations are needed for each split). For instance, a $10$-fold with $100$ 
different procedures would require at most $1000$ evaluations for a classical 
VF whereas we would need the computation of $49500$ tests for the TVF. It is already huge 
and does not even take into account the computation of the distances $h^2(\widehat{s}_{l,j},
\widehat{s}_{m,j})$, each one requiring the evaluation of an integral. Therefore a ``naive'' 
algorithm based on the computation of all the $V\times M$ values $\Dcal^2_j(m)$ would be very slow.

Fortunately, there is a smarter way to determine which $\widehat{m}$ minimizes $\overline{\Dcal}(\cdot)$
over $\Mcal$. Our algorithm is inspired in some way by the one described in \cite[Section 3]
{Magalhaes-Rozenholc2014}. In order to explain how this ``fast'' algorithm works, it will be 
convenient to single an element of $\Mcal$, that we shall denote by ``$m_s$'',  to serve as a starting 
point for our algorithm which begins with the computation of $\overline{\Dcal}(m_s)$. We store in
$R$ the minimal value of those $\overline{\Dcal}^2(m)$ that have already been computed and in $opt$ the 
corresponding optimal value of $m$ with initial values $opt=m_s$ and $R=\overline{\Dcal}^2(m_s)$. We update them after each computation of a new $\overline{\Dcal}^2(m)$ such that $\overline{\Dcal}^2(m)<R$, then setting $opt:=m$ and $R:=\overline{\Dcal}^2(opt)$ so that $R$ can only decrease during the computational procedure.

By (\ref{eq:critere TVF}), minimizing $\overline{\Dcal}^2(m)$ is equivalent to minimizing $\sum_{j=1}^V \Dcal^2_j(m)$. Since
\[
\Dcal^2_j(m)= \sup_{l\in\Mcal_m}h^2\paren{\widehat{s}_{l,j},\widehat{s}_{m,j}}\1_{\set{\psi_{l,m}(\Xbf)=l}}
\quad\mbox{ with }\quad\Mcal_m=\Mcal\setminus\{m\},
\]
one can compute it iteratively, starting with $\Lcal_j(m)=0$ and setting 
\[
\Lcal_j(m): =\max\paren{\Lcal_j(m),h^2\paren{\widehat{s}_{l,j},\widehat{s}_{m,j}}}\quad\mbox{when}\quad
\psi_{l,m}(\Xbf_j)=l\quad\mbox{for}\quad l\in\Mcal_m.
\]
If $\psi_{l,m}(\Xbf_j)=m$ we can instead update $\Lcal_j(l)$ by $\Lcal_j(l): =\max(\Lcal_j(l),
h^2(\widehat{s}_{l,j},\widehat{s}_{m,j}))$ using the result of the test $\psi_{l,m}(\Xbf_j)$ for the 
calculation of both $\Dcal_j^2(m)$ and $\Dcal_j^2(l)$. Our algorithm proceeds in this way, with a 
set of $M$ $V$-dimensional vectors $\Lcal_\cdot(m)$, $m\in\Mcal$, initially set to zero. The 
updating procedure of $\Lcal_j(m)$ stops when all updates, with $l\in\Mcal_m$, have been 
done (which means that the present value of $\Lcal_j(m)$ is $\Dcal^2_j(m)$) and we 
finally set $\overline{\Dcal}^2(m)=\sum_{j=1}^V\Lcal_j(m)$. 

We also use another trick in order to shorten our computations. Since $\Lcal_j(m)$ can only 
increase during the updating procedure, $\sum_{j=1}^V\Lcal_j(m)$ is, at any time, a lower bound 
for $\overline{\Dcal}^2(m)$, whatever $m\in\Mcal$. Therefore it is useless to go on with the computation of the 
vector $\Lcal_{\cdot}(m)$ if $\sum_{j=1}^V\Lcal_j(m)>R$ since then $\overline{\Dcal}^2(m)\ge\sum_{j=1}^V
\Lcal_j(m)$ cannot minimize the function $\overline{\Dcal}(\cdot)$ over $\Mcal$. Taking this fact into account,
we denote by $\Gcal\subset\Mcal$ the set of all procedures which are potentially ``better'' than the current optimal one stored in $opt$. This means that we store in $\Gcal$ all $m\in\Mcal$ for which 
we do not yet know whether $\overline{\Dcal}^2(m)<R$ or not and each time we find $m$ such that 
$\sum_{j=1}^V\Lcal_j(m)>R$, we remove it from $\Gcal$. We also remove $m$ from $\Gcal$ once
we have computed $\overline{\Dcal}^2(m)$ with $m\in\Gcal$ and then proceed with the computation of some 
new vector $\Lcal_\cdot(l)$ for $l\in\Gcal$ until $\Gcal$ is empty and the algorithm stops with the 
final value $\widehat{m}=opt$.\vspace{2mm}
\begin{algorithm}[H]\footnotesize
\caption{Selection of the TVF estimator}
\SetAlgoLined
\DontPrintSemicolon
\LinesNumbered
\mbox{}\\
\SetKwInput{KwInit}{\underline{Initialization}}
\KwInit{\null\\} 
Set $\Gcal=\Mcal_{m_s}$ and $opt=m_s$\\
\For{$(l\in\Mcal)$}{ 
		
		\For{$(j=1,\ldots,V)$}{
		
				$\Lcal_j(l)=0$
				}
		}\vspace{2mm}
\SetKwInput{KwInit}{\underline{$1$st step}}
\KwInit{\null\\} 
\For{$(l\in\Gcal)$}{
		
				 Compute $\psi_{m_s,l}(\Xbf_{j})$
				
				\eIf{$(\psi_{m_s,l}(\Xbf_{j})=m_s)$}{
					 $\Lcal_j(l)=h^2(\widehat{s}_{l,j},\widehat{s}_{m_s,j})$ 
				}{ 
					 $\Lcal_j(m_s)=\max(\Lcal_j(m_s), h^2(\widehat{s}_{l,j},\widehat{s}_{m_s,j}))$
		}
}
Set $R= \sum_{j=1}^V \Lcal_j(m_s)$ and $\Gcal=\Gcal\setminus 
\{ l\in \Gcal: \sum_{j=1}^V \Lcal_j(l)>R \}$\vspace{2mm}\\
\SetKwInput{KwInit}{\underline{Next steps}}
\KwInit{\null\\} 
\While{$(|\Gcal|>0)$}{
	 Choose $m\in \Gcal$ and set $\Gcal=\Gcal\setminus \{m\}$ \nllabel{alg: jump m}
	 
			\For{$(j=1,\ldots,V)$}{  

					\For{$(l\in\Mcal_m)$}{ 
								Compute $\psi_{m,l}(\Xbf_{j})$  \tcp*[f]{if it has not been done yet}
								
								\If{$(\psi_{m,l}(\Xbf_{j})=m$ and $l\in \Gcal)$}{ 
									 $\Lcal_j(l)=\max(\Lcal_j(l),h^2(\widehat{s}_{l,j},\widehat{s}_{m,j}))$
									 
									 \If{$( \sum_{i=1}^V \Lcal_i(l)>R)$}{ $\Gcal=\Gcal\setminus \{l\}$}
								}
								\If{$(\psi_{m,l}(\Xbf_{j})=l)$}{
										$\Lcal_j(m)=\max(\Lcal_j(m),h^2(\widehat{s}_{l,j},\widehat{s}_{m,j}))$ 
										
										\If{$(\sum_{i=1}^V \Lcal_i(m)>R)$\vspace{1mm}}{{\bf break} \tcp*[f]{quit the two ``for'' loops}}
								}
					}	
			}
	\If{$(\sum_{j=1}^V \Lcal_j(m)<R)$}{Set $opt=m$, $R= \sum_{j=1}^V \Lcal_j(m)$ and 
$\Gcal=\Gcal\setminus \{ l\in \Gcal:  \sum_{j=1}^V \Lcal_j(l)>R \}$}
}
 
\textbf{Return} $opt$
\end{algorithm}
\paragraph{Some important remarks}
\begin{itemize}
\item The algorithm is designed to work with any test procedure $\psi$ which satisfies Assumption (TEST) or, more generally, Assumption (TEST'), like the procedures based on the statistics \eqref{eq:test_birge} or \eqref{eq:test_baraud}.
\item It is important to notice that, at any step, we cannot ``delete'' once and for all the procedures which do not belong to the set $\Gcal$. Even if we do not compute the value of $\overline{\Dcal}$ for these procedures, we still need to test them against the remaining procedures in $\Gcal$.
\item We hoped that by starting from a good initial estimator, only a few procedures would be in the first 
set $\Gcal$, resulting in just a few tests. In the simulations we always started from $m_s=
\widehat{m}_{\rm LSVF}$ since the computation of $\widehat{m}_{\rm LSVF}$ is less costly than the one of $\widehat{m}_{\rm TVF}$ and provides a good starting point. If 
$\overline{\Dcal}(\widehat{m}_{\rm LSVF})=0$ at the first step the algorithm 
stops immediately and the chosen procedure is $\widehat{m}=\widehat{m}_{\rm LSVF}$. In this 
special case, the complexity of our algorithm is the same as the one of the classical approach.
\item Clearly, the choice of $m$ at line \ref{alg: jump m} of the algorithm, as well as the choice of the starting procedure, have no influence on the final estimator, only on the computational time. To avoid a quadratic complexity, we need to ensure that we don't ``jump'' to the worst procedure inside the set $\Gcal$ at each iteration. In our simulations, we chose to jump to the statistical method $k \in \Gcal$ with the lowest temporary criterion among the procedures in $\Gcal$, that is $k= \argmin_{l\in \Gcal} \sum_{j=1}^V \Lcal_j(l)$. We also tried two alternative options: jumping to $k= \argmax_{l\in \Gcal} \sum_{j=1}^V \Lcal_j(l)$ and to the most chosen statistical method $k$ in $\Gcal$ against $m$. Both options lead of course to the same final estimator but were definitely slower. 
\end{itemize}

\section{Proof of Theorem~\ref{T-kernel}}\label{P}
First note that the kernel estimator $\widehat{s}_w$ can be written, according to (\ref{Eq-kernest}),
$K_{w}\ast P_{n}$ where $P_{n}=n^{-1}\sum_{i=1}^{n}\delta_{X_{i}}$ denotes the empirical measure based on the i.i.d.\ sample $X_{1},...,X_{n}$. It then follows from the triangle inequality that
\begin{eqnarray}
\Bbb{E}_s\left[h^2\left(\widehat{s}_w,s\right)\right]&=&\frac{1}{2}\mathbb{E}_s\left\Vert \sqrt{K_{w}\ast P_{n}}-\sqrt{s}\right\Vert^{2}\nonumber\\&\le&\left\Vert \sqrt{s}-\sqrt{K_{w}\ast s}\right\Vert^{2}
+\mathbb{E}_s\left\Vert \sqrt{K_{w}\ast P_{n}}-\sqrt{K_{w}\ast s}\right\Vert^{2},
\label{Eq-hrisk}%
\end{eqnarray}
which is the usual bound of the risk as squared bias plus variance, and we shall bound both terms
successively.

\subsection{Bounding the bias\label{bias}}
It is well known that whenever the function $s$ belongs to $\mathbb{L}_{2}$, the quality of approximation of $s$ by the convolution $K_{w}\ast s$ depends on the modulus
of continuity $\omega_{2}(s,\cdot)$ of $s$ in $\mathbb{L}_{2}$ as given by (\ref{Eq-modcon}).
If we consider the Hellinger distance instead of the $\mathbb{L}_{2}$-distance it is expected that the quality of approximation should rather depend on the the modulus of continuity of $\sqrt{s}$ instead. 
The control of the bias term is provided by the following lemma:
%
\begin{lemma}\label{zbias}
Let $s$ and $K$ be some density functions with respect to
Lebesgue measure on the real line. Let $g=\sqrt{s}$ and assume that
$\omega_{2}(g,\eta)\leq\phi(\eta)$ for every
nonnegative $\eta$ and some nondecreasing and concave function $\phi$ on
$[0,\infty)  $ with $\phi(0)=0$. Then for every positive real number $w$
\begin{equation}
\left\Vert \sqrt{s}-\sqrt{K_{w}\ast s}\right\Vert^{2}\le
2\left[\int_{\mathbb{R}}\left(1\vee x^{2}\right)K(x)\,dx\right]\phi^{2}(w).
\label{ebias}
\end{equation}
\end{lemma}
\begin{proof}
The key of the proof is to compare $D^{2}=\left\Vert g-\sqrt{K_{w}\ast g^{2}}\right\Vert ^{2}$ with $\Delta
^{2}=\left\Vert g-K_{w}\ast g\right\Vert^{2}$. Our arguments are easier
to explain within a probabilistic framework. Let $\xi$ be some random variable
with density $K$ with respect to the Lebesgue measure. Then the convolution
operator can be written as
\[
(K_{w}\ast g)(x) =\mathbb{E}[g(x-w\xi)] \text{ for all }x\in\mathbb{R}.
\]
From Jensen's inequality we know that
\[
\mathbb{E}[g(x-w\xi)]\leq\sqrt{\mathbb{E}\left[g^{2}(x-w\xi)\right]},
\]
or equivalently $\sqrt{K_{w}\ast g^{2}}\ge K_{w}\ast g$, and a fortiori,
\begin{equation}
\int g(x)\sqrt{\left(K_{w}\ast g^{2}\right)(x)}\,dx\geq\int g(x)(K_{w}\ast g)(x)\,dx.
\label{escalarproduct}
\end{equation}
Expanding the square norms, we derive from (\ref{escalarproduct}) that
\[
D^{2}-\Delta^{2}\leq\left\Vert \sqrt{K_{w}\ast g^{2}}\right\Vert
^{2}-\left\Vert K_{w}\ast g\right\Vert^{2}.
\]
The trick is to notice that
\[
\left\Vert \sqrt{K_{w}\ast g^{2}}\right\Vert^{2}-\left\Vert K_{w}\ast
g\right\Vert^{2}=\int_{\mathbb{R}}\operatorname*{Var}\left(g(x-w\xi)\strut\right)dx.
\]
Now since the computation of the variance is not sensitive to the substraction of a constant
\[
\operatorname*{Var}\left(g(x-w\xi)\strut\right)
=\operatorname*{Var}\left(g(x-w\xi)-g(x)\strut\right)
\leq\mathbb{E}\left[\left(g(x-w\xi)-g(x)\strut\right)^{2}\right]
\]
and Fubini's Theorem implies that
\[
\left\Vert \sqrt{K_{w}\ast g^{2}}\right\Vert^{2}-\left\Vert K_{w}\ast
g\right\Vert^{2}\leq\mathbb{E}\left[\omega_{2}^{2}(g,w|
\xi|)\right].
\]
This achieves the first step of the proof. The second step is straightforward.
We just have to bound $\Delta^{2}$ which is an easy task since
\[
\Delta^{2}=\int_{\mathbb{R}}\left(\strut\mathbb{E}[g(x-w\xi)-g(x)]\right)^{2}dx
\]
implies by Jensen's inequality and Fubini's Theorem that
\[
\Delta^{2}\le\mathbb{E}\left[\int_{\mathbb{R}}\left(g(x-w\xi)
-g(x)\strut\right)^{2}dx\right]\le\mathbb{E}\left[\omega_{2}^{2}(g,w\vert \xi\vert)\right].
\]
Collecting these bounds, we derive that
\[
D^{2}\leq2\mathbb{E}\left[\omega_{2}^{2}(
g,w\vert \xi\vert)\right]\le2\mathbb{E}\left[\phi^{2}(w\vert \xi\vert)\right].
\]
It remains to decouple $w$ and $\xi$ in the last expression above. This can be
done by noticing that the monotonicity and concavity properties of $\phi$
imply that $\phi(w\vert \xi\vert)\le(1\vee\vert \xi\vert)\phi(w)$ and the result follows.
\end{proof}

\subsection{The variance term}
We now turn to the analysis of the variance term of the Hellinger risk of a kernel
estimator. 
\begin{lemma}\label{zvarhel}
Let us denote by $A_{w}$ the Borel set $\{x\in\mathbb{R}\,|\,(K_{w}\ast s)(x)>0\}$, then
\begin{equation}
\mathbb{E}_s\left[\left\Vert \sqrt{K_{w}\ast P_{n}}-\sqrt{K_{w}\ast s}\right\Vert
^{2}\right]\le\frac{1}{nw}\int_{A_{w}}\frac{\left((K^{2})_{w}\ast
s\right)\!(x)}{(K_{w}\ast s)(x)}\,dx.
\label{evar1}
\end{equation}
In particular if $s$ is supported on an interval of finite length $2L$, then
\begin{equation}
\mathbb{E}_s\left[\left\Vert \sqrt{K_{w}\ast P_{n}}-\sqrt{K_{w}\ast s}\right\Vert^{2}\right]\le\frac{1}{nw}
\int\sup_{-L\le z\le L}K\left(\frac{x-z}{w}\right)dx.
\label{evar2}
\end{equation}
If the kernel $K$ is bounded, non-decreasing on $(-\infty,M_1)$ and non-increasing on $(M_2,+\infty)$
with $M_1\le M_2$, then
\begin{eqnarray}
\lefteqn{\int\sup_{-L\le z\le L}K\left(\frac{x-z}{w}\right)dx}\hspace{20mm}\nonumber\\&\le&2L\|K\|_\infty
+w\left[(M_2-M_1)\|K\|_\infty+\int_{-\infty}^{M_1}K(x)\,dx+\int_{M_2}^{\infty}K(x)\,dx\right].\qquad
\label{hvar}
\end{eqnarray}
If, in particular, $K$ is unimodal, then
\[
\mathbb{E}_s\left[\left\Vert \sqrt{K_{w}\ast P_{n}}-\sqrt{K_{w}\ast s}\right\Vert^{2}\right]\le\frac{2L\|K\|_\infty}{nw}
+\frac{1}{n}.
\]

\end{lemma}
\begin{proof}
Since, for $u,v\ge0$,
\[
(u-v)^2=\left(\sqrt{u}-\sqrt{v}\right)^2\left(\sqrt{u}+\sqrt{v}\right)^2\ge v\left(\sqrt{u}-\sqrt{v}\right)^2,
\]
it follows that
\[
\int\left(\sqrt{u(x)}-\sqrt{v(x)}\right)^2dx\le\int_{v(x)>0}\frac{[u(x)-v(x)]^2}{v(x)}\,dx+\int_{v(x)=0}u(x)\,dx,
\]
hence
\[
\left\Vert \sqrt{K_{w}\ast P_{n}}-\sqrt{K_{w}\ast s}\right\Vert^{2}\le\int_{A_w}\frac{\left[(K_{w}\ast P_{n})(x)-(K_{w}\ast s)(x)\right]^2}{(K_{w}\ast s)(x)}\,dx+\int_{A_w^c}(K_{w}\ast P_{n})(x)\,dx.
\]
By Fubini and the definition of $A_w$,
\[
\Bbb{E}_s\left[\int_{A_w^c}(K_{w}\ast P_{n})(x)\,dx\right]=\int_{A_w^c}\mathbb{E}_s\left[(K_{w}\ast P_{n})(x)\right]dx=\int_{A_w^c}(K_{w}\ast s)(x)\,dx=0.
\]
Taking expectations and using Fubini again, we therefore get
\begin{eqnarray*}
\mathbb{E}_s\left[\left\Vert \sqrt{K_{w}\ast P_{n}}-\sqrt{K_{w}\ast s}\right\Vert^{2}\right]&\le&
\int_{A_w}\frac{\mathbb{E}_s\left[\left[(K_{w}\ast P_{n})(x)-K_{w}\ast s(x)\right]^2\right]}{(K_{w}\ast s(x)}\,dx
\\&=&\int_{A_w}\frac{\operatorname*{Var}\left(\strut(K_{w}\ast P_{n})(x)\right)}{K_{w}\ast s(x)}\,dx
\end{eqnarray*}
and (\ref{evar1}) follows since
\begin{equation}
\operatorname*{Var}\left(\strut(K_{w}\ast P_{n})(x)\right)
=\frac{\operatorname*{Var}\left(\strut K_{w}(x-X)\right)}{n}
\leq\frac{\Bbb{E}_s\left[\left(\strut K_{w}(x-X)\right)^2\right]}{n}=\frac{\left((K^{2})_{w}\ast s\right)\!(x)}{nw}.
\label{evarx}
\end{equation}
Now observe that if $f$ is supported on $[a,a+2L]$,
\[
\left((K^{2})_{w}\ast s\right)\!(x)=\int_a^{a+2L}\frac{1}{w}K^2\left(\frac{x-z}{w}\right)s(z)\,dz
\le\sup_{a\le z\le a+2L}K\left(\frac{x-z}{w}\right)(K _{w}\ast s)(x),
\]
so that by (\ref{evar1})
\[
\mathbb{E}_s\left[\left\Vert \sqrt{K_{w}\ast P_{n}}-\sqrt{K_{w}\ast s}\right\Vert
^{2}\right]\le\frac{1}{nw}\int\sup_{a\le z\le a+2L}K\left(\frac{x-z}{w}\right)dx.
\]
Since
\[
\int\sup_{a\le z\le a+2L}K\left(\frac{x-z}{w}\right)dx=\int\sup_{-L\le y\le L}K\left(\frac{x-a-L-y}{w}\right)dx
=\int\sup_{-L\le y\le L}K\left(\frac{v-y}{w}\right)dv,
\]
(\ref{evar2}) follows. 

If $K$ is nonincreasing on $(M_2,+\infty)$ and $x>M_2w+L$, then
$\sup_{-L\le y\le L}K\left(w^{-1}(x-y)\right)=K\left(w^{-1}(x-L)\right)$ and
\[
\int_{M_2w+L}^{\infty}\,\sup_{-L\le y\le L}K\left(\frac{x-y}{w}\right)dx=
\int_{M_2w+L}^{\infty}K\left(\frac{x-L}{w}\right)dx=w\int_{M_2}^{\infty}K(y)\,dy.
\]
Similarily,
\[
\int_{-\infty}^{M_1w-L}\sup_{-L\le y\le L}K\left(\frac{x-y}{w}\right)dx=
w\int_{-\infty}^{M_1}K(y)\,dy
\]
and finally
\[
\int\sup_{-L\le y\le L}K\left(\frac{x-y}{w}\right)dx\le(M_2w+L-M_1w+L)\|K\|_\infty+
w\int_{-\infty}^{M_1}K(x)\,dx+w\int_{M_2}^{\infty}K(x)\,dx,
\]
which is (\ref{hvar}). The unimodal case immediately follows from (\ref{evar2}) and (\ref{hvar}) with $M_1=M_2$.
\end{proof}
\paragraph{Acknowledgments} 
One  author thanks Mathieu Sart for his helpful comments on an earlier version of the paper.

\bibliographystyle{plain} 
\bibliography{biblio} 

\begin{thebibliography}{10}

\bibitem{Arlot-Celisse2010}
S.~Arlot and A.~Celisse.
\newblock A survey of cross-validation procedures for model selection.
\newblock {\em Statistics Surveys}, 4:40--79, 2010.

\bibitem{Arlot-Lerasle2014}
S.~Arlot and M.~Lerasle.
\newblock Why \textsc{V} = 5 is enough in \textsc{V}-fold cross-validation.
\newblock {\em arXiv:1210.5830v2}, 2014.

\bibitem{Baraud2011}
Y.~Baraud.
\newblock Estimator selection with respect to \textsc{H}ellinger-type risks.
\newblock {\em Probab. Theory Related Fields}, 151:353--401, 2011.

\bibitem{Berlinet-Devroye1994}
A.~Berlinet and L.~Devroye.
\newblock A comparison of kernel density estimates.
\newblock {\em Publications de l'Institut de Statistique de l'Universit\'{e} de
  Paris}, 38(3):3--59, 1994.

\bibitem{Birge1983}
L.~Birg\'{e}.
\newblock Approximation dans les espaces m\'{e}triques et th\'{e}orie de
  l'estimation.
\newblock {\em Z. Wahrscheinlichkeitstheorie verw. Geb.}, 65:181--237, 1983.

\bibitem{Birge1984b}
L.~Birg\'{e}.
\newblock Stabilit\'{e} et instabilit\'{e} du risque minimax pour des variables
  ind\'{e}pendantes \'{e}quidistribu\'{e}es.
\newblock {\em Ann. Inst. H. Poincar\'{e} Sect. B}, 20:201--223, 1984.

\bibitem{Birge2006b}
L.~Birg\'{e}.
\newblock Model selection via testing: an alternative to (penalized) maximum
  likelihood estimators.
\newblock {\em Ann. Institut Henri Poincar\'{e}, Probab. et Statist.},
  42:273--325, 2006.

\bibitem{Birge2013a}
L.~Birg\'{e}.
\newblock Robust tests for \textsc{M}odel \textsc{S}election.
\newblock {\em From Probability to Statistics and Back: High-Dimensional Models
  and Processes -- A Festschrift in Honor of Jon A. Wellner (M.Banerjee, F.
  Bunea, J. Huang, V. Koltchinskii and M. Mathuis,eds)}, IMS Collections --
  Volume 9:47--64, 2013.

\bibitem{Birge-Massart1993}
L.~Birg\'e and P.~Massart.
\newblock Rates of convergence for minimum contrast estimators.
\newblock {\em Probab. Theory Related Fields}, 97:113--150, 1993.

\bibitem{Birge-Rozenholc2006}
L.~Birg\'{e} and Y.~Rozenholc.
\newblock How many bins should be put in a regular histogram.
\newblock {\em ESAIM Probab. Statist.}, 10:24--45, 2006.

\bibitem{Devroye-Lugosi2001}
L.~Devroye and G.~Lugosi.
\newblock {\em Combinatorial Methods in Density Estimation.}
\newblock Springer-Verlag, New York, 2001.

\bibitem{LeCam1973}
L.M. Le~Cam.
\newblock Convergence of estimates under dimensionality restrictions.
\newblock {\em Ann. Statist.}, 1:38--55, 1973.

\bibitem{LeCam1975}
L.M. Le~Cam.
\newblock On local and global properties in the theory of asymptotic normality
  of experiments.
\newblock {\em In Stochastic Processes and Related Topics, Academic Press, New
  York.}, 1:13--54, 1975.

\bibitem{Magalhaes-Rozenholc2014}
N.~Magalh\~aes and Y.~Rozenholc.
\newblock An efficient algorithm for \textsc{T}-estimation.
\newblock {\em http://hal.archives-ouvertes.fr/hal-00986229}, 2014.

\bibitem{Marron1987}
J.S. Marron.
\newblock A \textsc{C}omparison of \textsc{C}ross-\textsc{V}alidation
  \textsc{T}echniques in \textsc{D}ensity \textsc{E}stimation.
\newblock {\em Ann. Statist.}, 15(1):152--162, 1987.

\bibitem{Massart2007}
P.~Massart.
\newblock {\em Concentration Inequalities and Model Selection.}
\newblock Lecture on Probability Theory and Statistics. Ecole d'Et\'{e} de
  Probabilit\'{e}s de Saint-Flour XXXIII - 2003 (J. Picard, ed.) Lecture Notes
  in Math. Springer, Berlin, 2007.

\bibitem{Mildenberger-Weinert2012}
T.~Mildenberger and H.~Weinert.
\newblock The benchden package: Benchmark densities for nonparametric density
  estimation.
\newblock {\em Journal of Statistical Software}, 46(14):1--14, 2012.

\bibitem{Rudemo1982}
M.~Rudemo.
\newblock Empirical \textsc{C}hoice of \textsc{H}istograms and \textsc{K}ernel
  \textsc{D}ensity \textsc{E}stimators.
\newblock {\em Scandinavian Journal of Statistics.}, 9(2):65--78, 1982.

\bibitem{Walter-Blum1979}
G.~Walter and J.~Blum.
\newblock Probability density estimation using delta sequences.
\newblock {\em Ann. Statist.}, 7:328--340, 1979.

\bibitem{Watson-Leadbetter1965}
G.S. Watson and M.R. Leadbetter.
\newblock Hazard analysis \textsc{II}.
\newblock {\em Sankhya Ser. A.}, 26:101--116, 1965.

\bibitem{Whittle1958}
P.~Whittle.
\newblock On the smoothing of probability density functions.
\newblock {\em J. Roy. Statist. Soc. Ser. B.}, 20:334--343, 1958.

\bibitem{Winter1975}
B.B. Winter.
\newblock Rate of strong consistency of two nonparametric density estimators.
\newblock {\em Ann. Statist.}, 3(3):759--766, 1975.

\end{thebibliography}

\appendix
\section{Supplementary material}\label{sec:SuppMat}
We provide here additional simulations about the TVF based on the test statistic $T_{t,u}(\Xbf)$ designed by Baraud in\cite{Baraud2011} and given by \eqref{eq:test_baraud}. As in \secref{Empirical study}, we study the influence of $V$ and we compare the TVF based on this test with classical VF procedures. The results are summarized in Table~\ref{tab:influence_V_Baraud} and Figure~\ref{fig:TVF_Baraud vs contrastVF (hellinger)} which are the analogues of Table~\ref{tab:influence_V_Birge} and Figure~\ref{fig:TVF_Birge vs contrastVF (hellinger)} respectively.\vspace{2mm}
\begin{table}[h]
\centering
\small
\begin{tabular}{|l|c|r|r|r|r|r|r|r|r|r|r|r|}
\hline
family&$V$&$s_1$&$s_2$&$s_3$&$s_4$&$s_5$&$s_7$&$s_{12}$&$s_{13}$&$s_{22}$&$s_{23}$&$s_{24}$\\\hline
\multicolumn{1}{|c|}{\multirow{5}{*}{$\Fcal_{\rm R}$}}&$2$&\textbf{2,89}&9,97&9,07&13,2&10,5	&11  &17,5&14,7	&10,3	&19,9	&26,9\\
\multicolumn{1}{|c|}{}&$5$	&4,33&9,68	&8,61	&12,4&9,87	&10,4&17,1	&\textbf{13,4}	&9,37	&17,8	&24,7\\
\multicolumn{1}{|c|}{}&$10$	&6,13&9,65	&8,56	&12,1&9,65	&10,4&17  	&13,7	&9,36	&17,5	&\textbf{24,3}\\
\multicolumn{1}{|c|}{}&$20$	&9,28&\textbf{9,47}	&\textbf{8,4}&\textbf{12}&\textbf{9,36}	&\textbf{10,3}&\textbf{16,9}	&14,2	&\textbf{9,17}	&\textbf{17,4}	&24,6\\\cline{2-13}
\multicolumn{1}{|c|}{}&BR&2,20&9,94&9,27&12,98&10,53&11,14&17,85&14,63&10,37&17,98&25,15\\
\hline
\hline
\multicolumn{1}{|c|}{\multirow{5}{*}{$\Fcal_{\rm K}$}}&$2$&15,6&29,4&5,69&5,07&\textbf{3,55}	&4,24&27,2	&20 &3,97&10,3&18  \\
\multicolumn{1}{|c|}{}&$5$&13,2&25,7&5,1&\textbf{4,94}&3,58	&\textbf{3,97}&23 	&18,1	&\textbf{3,85}	&9,18	&16,2\\
\multicolumn{1}{|c|}{}&$10$&12,9&24,8&5  &5,02&3,86	&4,01&22,2	&17,7	&3,87	&9,04	&\textbf{15,8}\\
\multicolumn{1}{|c|}{}&$20$&\textbf{12,7}&\textbf{24,4}&\textbf{4,98}&5,28&4,54	&4,1 &\textbf{21,6}	&\textbf{17,6}&3,98	&\textbf{8,98}&\textbf{15,8}\\\cline{2-13}
\multicolumn{1}{|c|}{}&UCV&15,86&22,20&5,57&6,16&3,74&4,10&18,80&17,16&3,88&9,52&15,91\\
\hline
\hline
\multicolumn{1}{|c|}{\multirow{4}{*}{$\Fcal_{\rm KR}$}}&$2$&\textbf{2,87}&10&7,47&5,88&5,04&5,6&18,9&14,7	&6,38	&11,6	&19,1\\
\multicolumn{1}{|c|}{}&$5$&3,68&\textbf{9,77}&6,81&\textbf{5,48}&\textbf{4,64}&\textbf{5,19}&17,7&\textbf{13,3}&\textbf{5,01}&9,3&16,4\\
\multicolumn{1}{|c|}{}&$10$&3,58&9,84&6,71&5,53&4,99&5,26&\textbf{17,6}	&13,7&5,11&9,04&15,9\\
\multicolumn{1}{|c|}{}&$20$&3,79&9,84&\textbf{6,45}	&5,65&5,31	&5,83&\textbf{17,6}	&14,6&5,22&\textbf{9,01}&\textbf{15,7}\\
\hline
\end{tabular}  
\caption{\label{tab:influence_V_Baraud} \small 1000 times Hellinger risks for the TVF procedure based on Baraud's test.}
\end{table}

\begin{figure}[h!]
\[ \includegraphics[height=3cm]{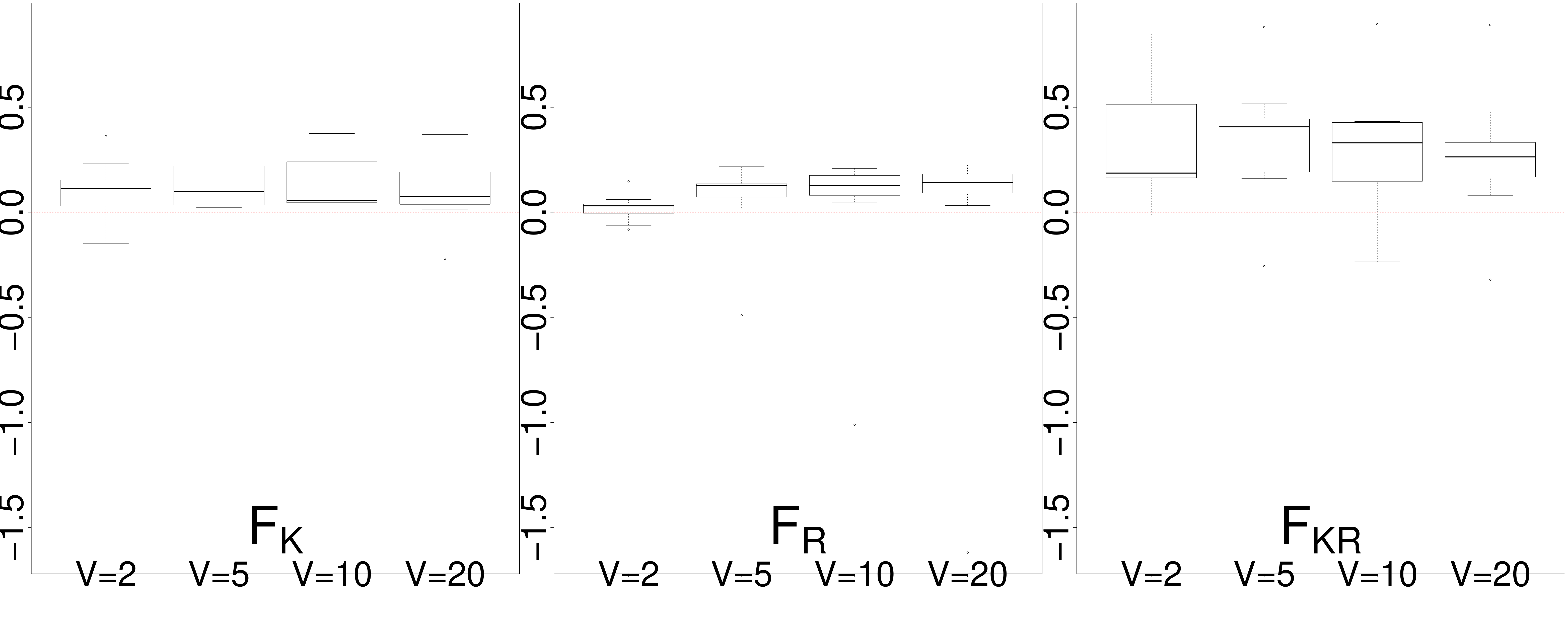} \]
\[ \includegraphics[height=3cm]{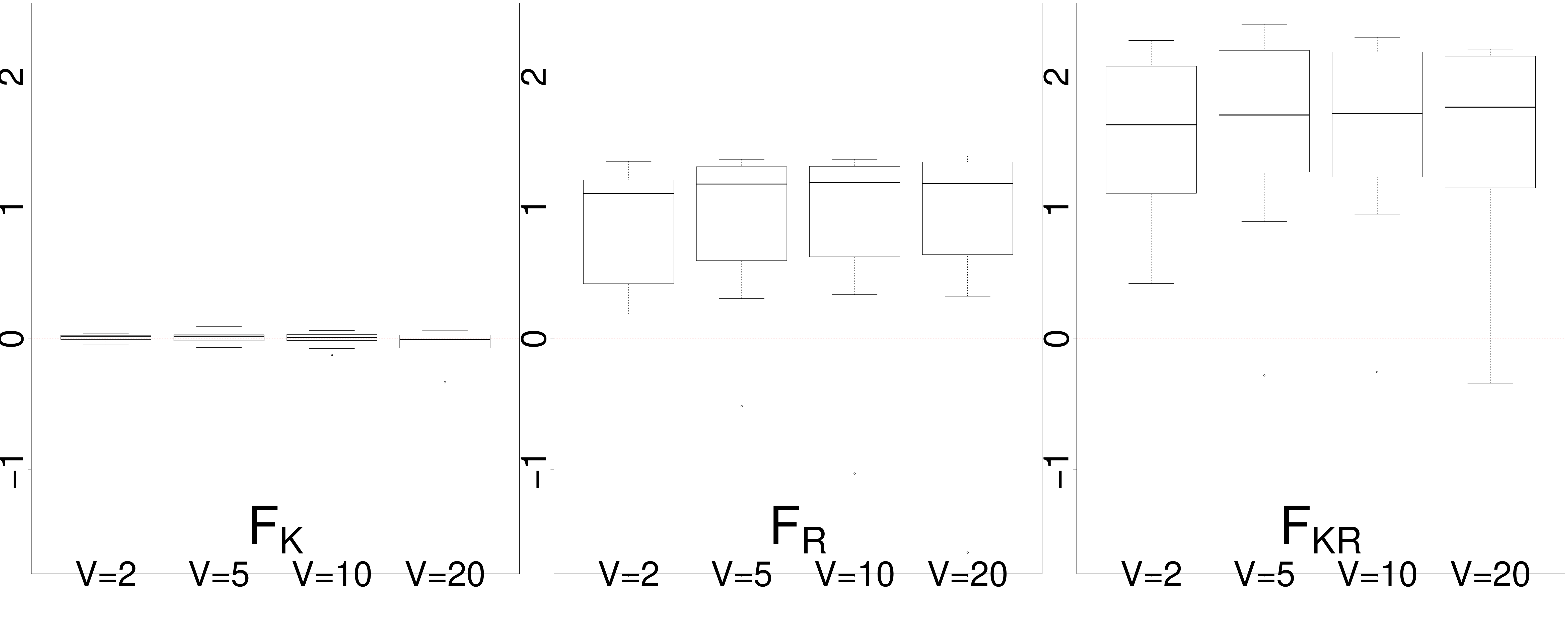} \]
\caption{\label{fig:TVF_Baraud vs contrastVF (hellinger)} \small From left to right, the boxplots of $\overline{W}_s\paren{\tilde s,\tilde s_{\rm TVF},h^2}$ using families $\Fcal_{\rm K},\Fcal_{\rm R}$ and $\Fcal_{\rm KR}$ (up for $\tilde s=\widehat{s}_{\widehat{m}_{{\rm LSVF}}}$, down for $\tilde s=\widehat{s}_{\widehat{m}_{{\rm KLVF}}}$). Each subfigure shows the boxplot for $V=2,5,10$ and 20. The horizontal red dotted line provides the reference value 0.}\vspace{2mm}
\end{figure}

\subsection*{Influence of the test on the TVF}\label{sub:ComparisonTest}
We compare here the performances of the best TVF procedure (among the five values of $\theta$ described above) derived from Birg\'e's test \eqref{eq:test_birge} against the one deduced from Baraud's test \eqref{eq:test_baraud} (denoted $\widehat{s}_{\widehat{m}_{\rm TVF}}$). We show the conclusion of our study for the families $\Fcal_{\rm R}$, $\Fcal_{\rm K}$ and $\Fcal_{\rm KR}$, $n=500$ and $V=2,5,10$ and 20. The results are very similar for other values of $n$. For the sake of clarity and to emphasize the similarity of both procedures in terms of Hellinger risk, we present for each family, for each $V$, the supremum and the infimum over $\Lcal$ of the ratio
\[ \Upsilon(s)=\set{\inf_{\theta\in \Theta} \overline{R}_n\paren{\widehat{s}_{\widehat{m}(\theta)},s,h^2} \Big/ \overline{R}_n\paren{\widehat{s}_{\widehat{m}_{\rm TVF}},s,h^2} }. \] 
If $\inf_{s\in \Lcal} \Upsilon(s)\ge1$ the TVF using Baraud's test behaves in a better way than the one using Birg\'e's test for all densities in $\Lcal$ while if $\sup_{s\in \Lcal} \Upsilon(s)\le1$ the opposite holds. The closer the two values, the more similar the quality of both procedures.

\begin{table}[h]
\centering
\small
\begin{tabular}{|l|c|r|r|r|r|}
\hline
family&$\Upsilon(s)$&$V=2$&$V=5$&$V=10$&$V=20$\\\hline
\multicolumn{1}{|c|}{\multirow{2}{*}{$\Fcal_{\rm R}$}}&$\sup_s$&103,68&102,59&101,72&102,27 \\
\multicolumn{1}{|c|}{}&$\inf_s$&98,16&100,07&99,59&99,13 \\\hline
\multicolumn{1}{|c|}{\multirow{2}{*}{$\Fcal_{\rm K}$}}&$\sup_s$&102,78&100,80&100,92&105,10\\
\multicolumn{1}{|c|}{}&$\inf_s$&99,58&98,72&97,45&96,13\\\hline
\multicolumn{1}{|c|}{\multirow{2}{*}{$\Fcal_{\rm KR}$}}&$\sup_s$&116,71&115,80&116,79&116,73\\
\multicolumn{1}{|c|}{}&$\inf_s$&96,70&98,84&99,08&99,30\\
\hline
\end{tabular}
\caption{\label{tab:testLB_testYB} \small Supremum and infimum of 100 times the ratio, see the text.}
\end{table}

We see from this table that Baraud's and Birg\'e's test are very similar to process the TVF procedure for families $\Fcal_{\rm R}$ and $\Fcal_{\rm K}$. There is indeed no noticeable difference for these families, the largest gain (for a density in $\Lcal$) being of 5\% only.
The procedure based on Baraud's test becomes much better for the family $\Fcal_{\rm KR}$. We observe indeed that a potential gain of 15\% appears (since the $\sup_s$ is close to 115\%) while the loss is negligible (since the $\inf_s$ is close to 99\%). 
Moreover, the ratios are quite similar when $V$ increases.
Finally, let us recall that the TVF procedure based on \eqref{eq:test_birge} is less time-consuming since it requires to compute only one integral instead of two for \eqref{eq:test_baraud}.
\end{document}